\documentclass[12pt,reqno]{amsart}
\usepackage{amsmath,amsthm,amsfonts,amssymb,amscd,amstext}
\usepackage[latin1]{inputenc}
\usepackage[pdftex]{hyperref}
\usepackage[dvips]{graphicx}
\usepackage{psfrag,euscript,xcolor}
\usepackage{a4wide}

\newcommand{\qand}{\quad\text{and}\quad}

\theoremstyle{plain}
\newtheorem{maintheorem}{Theorem}

\newtheorem{theorem}{Theorem}[section]
\newtheorem{proposition}[theorem]{Proposition}
\newtheorem{corollary}[theorem]{Corollary}
\newtheorem{lemma}[theorem]{Lemma}

\newtheorem{conjecture}{Conjecture}

\theoremstyle{definition}
\newtheorem{remark}[theorem]{Remark}
\newtheorem{definition}{Definition}
\newtheorem{example}{Example}

\newcommand{\dem}{\begin{proof}}
\newcommand{\cqd}{\end{proof}}

\def \epsilon {\varepsilon}

\def \wt {\widetilde}
\def \un{\underline }
\def \ov{\overline }

\newcommand{\cC}{\EuScript{C}}

\newcommand{\cH}{\EuScript{H}}

\newcommand{\cO}{\EuScript{O}}
\newcommand{\cP}{\EuScript{P}}

\newcommand{\cS}{\EuScript{S}}

\newcommand{\W}{\EuScript{W}}

\newcommand \cc {{\mathcal C}}
\newcommand \cK {{\mathcal K}}

\newcommand \cA {{\mathcal A}}
\newcommand \cW {{\mathcal W}}
\newcommand \cG {{\mathcal G}}
\newcommand \cI {{\mathcal I}}

\newcommand \vfi {\varphi}

\newcommand{\fB}{{\mathfrak B}}

\newcommand{\RR}{{\mathbb R}}

\newcommand{\NN}{{\mathbb N}}
\newcommand{\PP}{{\mathbb P}}

\newcommand{\XX}{{\mathbb X}}
\newcommand{\YY}{{\mathbb Y}}
\newcommand{\ZZ}{{\mathbb Z}}

\newcommand{\interior}{\operatorname{inter}}

\newcommand{\leb}{\operatorname{Leb}}

\newcommand{\supp}{\operatorname{supp}}
\newcommand{\per}{\operatorname{Per}}
\newcommand{\diam}{\operatorname{diam}}

\author{V. Araujo and V. Pinheiro}

\address{V. Araujo, Departamento de Matem\'atica, Universidade Federal da Bahia\\  Av. Ademar de Barros s/n, 40170-110 Salvador, Brazil.}

\email{vitor.d.araujo@ufba.br or vitor.araujo.im.ufba@gmail.com}

\address{V. Pinheiro, Departamento de Matem\'atica,
  Universidade Federal da Bahia, Av. Ademar de Barros s/n,
  40170-110 Salvador, Brazil.}

\email{viltonj@ufba.br}



\date{\today}

\thanks{Work carried out at the Federal University of Bahia,
  and IMPA.  This study was financed in part by the
  Coordenação de Aperfeiçoamento de Pessoal de Nível
  Superior - Brasil (CAPES) - Finance Code 001; CNPq Project
  301392/2015-3; FAPESB Project PIE0034/2016; IMPA and
  PRONEX - Dynamical Systems.}

\title{Abundance of wild historic behavior}

\begin{document}

\begin{abstract}
  Using Caratheodory measures, we associate to each positive
  orbit $\cO_{f}^{+}(x)$ of a measurable map $f$, a Borel
  measure $\eta_{x}$. We show that $\eta_{x}$ is
  $f$-invariant whenever $f$ is continuous or $\eta_{x}$ is
  a probability. 
  These measures are used to
  study the \emph{historic} points of the system, that is,
  \emph{points with no Birkhoff averages}, and we construct
  topologically generic subset of \emph{wild historic
    points} for wide classes of dynamical models. 
  We use properties of the measure
  $\eta_x$ to deduce some features of the dynamical system
  involved, like the \emph{existence of heteroclinic
    connections from the existence of open sets of historic
    points}.
\end{abstract}

\maketitle
\tableofcontents

\section{Introduction}
\label{sec:introduction}

The asymptotic behavior of a dynamical system given by some
transformation $f:M\circlearrowleft$ is in general quite
complex. To understand the behavior of the orbit of a point
$x$ we should focus on the $\omega$-limit set $\omega(x)$,
the set of accumulation points of $f^n(x), n\ge1$. The
number of iterates of the positive orbit of $x$ near some
given point $y\in\omega(x)$ depends on $y$ and, in fact, the
limit frequency of visits of the orbit of $x$ to some subset
$A$ of the phase space $M$ is non-existent in general! 

A simple and very general example is obtained taking a
non-periodic point $x$, an strictly increasing integer
sequence $n_j=\sum_{i=0}^j 10^j$ and the subset
\begin{align*}
  A=\{f^{n_1}(x),\dots,f^{n_2}(x)\}\cup\{f^{n_3}(x),\dots,f^{n_4}(x)\}\cup\dots
  =\cup_{i\ge0}\{f^{n_{2i+1}}(x),\dots,f^{n_{2(i+1)}}(x)\}.
\end{align*}
Then it is straightforward to check that
\begin{align*}
  \frac1{n_k}\sum_{j=1}^{n_k} \delta_{f^j(x)}(A)
  \begin{cases}
    \ge 9/10 & \text{if $k$ is even}
    \\
    \le 1/10 & \text{if $k$ is odd}
  \end{cases}.
\end{align*}
Thus even if $\lim_k f^k(x)$ exists, we still have that
$\lim_k \frac1{n_k}\sum_{j=1}^{n_k} \delta_{f^j(x)}(A)$ does
not exist, In Ergodic Theory behavior of this type is
bypassed through the use of weak$^*$ convergence in
Birkhoff's Ergodic Theorem, ensuring that $\lim_k
\frac1{n_k}\sum_{j=1}^{n_k} \delta_{f^j(x)}$ exists in the
weak$^*$ topology for almost every point with respect to any
$f$-invariant measure.

As defined by Ruelle in \cite{ruelle2001} and Takens in
\cite{takens08} we say that \emph{$x$ has ``historic
  behavior''} if the sequence $\frac1{n_k}\sum_{j=1}^{n_k}
\delta_{f^j(x)}$ does not converge in the weak$^*$ topology.

The above construction of non-convergent sequence of
frequency of visits can be easily obtained in any
topological Markov chain and so, through Markov partitions
and their coding, can also be obtained in every hyperbolic
basic set of an Axiom A diffeomorphism; see
e.g. Bowen and Ruelle in \cite{BR75}.

As shown by Takens \cite{takens08} (and also Dowker
\cite{Dowk53}) the set of points with historic behavior is
residual.  Hence, generically a point of a basic piece of
an Axiom A diffeomorphism has \emph{both dense orbit and
  historic behavior!} More recently,
Liang-Sun-Tian~\cite{LiSuTi2013} extend this result for the
support of non-uniformly hyperbolic measures for smooth
diffeomorphisms and Kiriki-Ki-Soma \cite{KikiLiSo15} obtain a
residual historic subset in the basin of attractor of any
geometric Lorenz model. Here we extend and strengthen these
results to whole classes of dynamical models.

We say that a measure $\mu$ which gives positive mass to the
subset of points with historic behavior is a \emph{historic
  measure}.

The construction of the heteroclinic attractor attributed to Bowen,
presented in \cite{Ta95}, as a flow with an open subset of points with
historic behavior, provides an example of such phenomenon where
Lebesgue measure is a non-invariant historic measure; see
Example~\ref{ex:Boweneye} in
Section~\ref{sec:additive-measure-eta}. Other similar examples can be
found in Gaunersdorfer \cite{gaunersdorfer1992}.  In the quadratic
family, a result from Hofbauer and Keller~\cite{HK90} provides many
other examples. Indeed, these authours proved that for the family
$f_t:[0,1]\circlearrowleft, f_t(x)=4t x(1-x)$, there exists a
non-denumerable subset of parameters $t$ in $[0,1]$ such that Lebesgue
measure is an historic measure with respect to $f_t$.

  We emphasize that Jordan, Naudot and Young showed in
  \cite{JNY09}, through a classical result from Hardy
  \cite{Hardy}, that \emph{if time averages
    $\frac1n\sum_{j=0}^{n-1}\phi(f^jx)$ of a bounded
    observable $\phi:\XX\to\RR$ do not converge}, then
  \emph{all higher order averages do not exist either, that
    is, every Césaro or H\"older higher order means fail to
    converge;} see e.g. \cite{Hard92} for the definitions of
  these higher order summation processes. This shows that
  \emph{historic points cannot be regularized by taking higher
  order averages.}

This paper aims at studying the ergodic properties of
historic measures. For that we consider a set function
$\tau_x(A)=\limsup_{n\to\infty} \frac1{n}\sum_{j=1}^{n}
\delta_{f^j(x)}(A)$ that will serve as a pre-measure to
obtain a Borel measure $\eta_x$ through a classic well-known
construction of Carathéodory; a thorough presentation of
which can be found in Rogers \cite{rogers}.  
This is a way to bypass failure of convergence of time
averages in certain classes of systems, in particular in the
``Bowen eye'' example.





We show that if $\eta_x$ is purely atomic
for an open subset of points $x$, then the system has
similar dynamics to the heteroclinic Bowen attractor; see
\cite{Ta95,gaunersdorfer1992}. In particular, Lebesgue
measure is historic.


Many more results about the set of points with historic
behavior are known. To the best of our knowledge, Pesin and
Pitskel were the first to show that these points carry full
topological pressure and satisfy a variational principle for
full shifts, in \cite{PesPits84}.  These points are named
``non-typical'' in \cite{BarSch00} by Barreira-Schmeling who
show that this set has full Hausdorff dimension and full
topological entropy for subshifts of finite type, conformal
repellers and conformal horseshoes.  These points are said
``irregular points'' in the work \cite{thomp10} of Thompson
where it is shown that, for maps with the specification
property having some point with historic behavior, the set
of such irregular points carries full topological
pressure. This indicates that this set of points is not
dynamically irrelevant.  Genericity of these ``irregular
points'' for dynamics with specification and non-convergence
of time averages for an open and dense family of continuous
functions follows from \cite[Proposition 21.18]{DGS76}; see
Lemma~\ref{le:nonconst} in Subsection
\ref{sec:non-existence-time} and also
\cite{BarrLiValls16,BarrLiValls14a,BarrLiValls14,BarrLiVals12}.

The Hausdorff dimension of subsets of such points is studied
by Barreira-Saussol \cite{BarrSau01} for hyperbolic sets for
smooth transformations and also by Olsen-Winter in
\cite{olsen-winter07}, in the setting of multifractal
analysis for subshitfs of finite type and, for several
specific transformations, also by Olsen
\cite{olsen2004a,olsen2004}, extended by Zhou-Chen
\cite{ZhCh13} and generalized by Tian-Varandas
\cite{VarXue}. For a deeper multifractal analysis of the set
of historic points, see the recent works of Bomfim and
Varandas
\cite{BomfimVarandas15a,BomfimVarandas15}. 

Recently Kiriki and Soma \cite{KrS15} show
that every $C^2$ surface diffeomorphism exhibiting a generic
homoclinic tangency is accumulated by diffeomorphisms which
have non-trivial wandering domains whose forward orbits have
historic behavior.
More recently Laboriau and Rodrigues \cite{LabRod} motivated
by the ideas in \cite{KrS15} present an example of a
parametrized family of flows admitting a dense subsets of
parameter values for which the set of initial conditions
with historic behaviour contains an open set. These are
examples of persistent historic behavior.

In this paper we also provide broad classes of examples
where $\eta_x$ has infinite mass on every open subset. These
\emph{points with wild historic behavior} are also
topologically generic for Axiom A systems, for expanding
measures, topological Markov chains, suspension semiflows
over these maps, Lorenz-like and Rovella-like attractors and
many other dynamical models. These are the ``points with
maximal oscillation'' studied by Denker, Grillenberger and
Sigmund in~\cite[Proposition 21.18]{DGS76} for continuous
dynamical systems in metric spaces admitting specification;
see Lemma~\ref{le:wildmaxosc} in
Subsection~\ref{sec:wild-historic-points-1}.  Olsen in
\cite{Ols04,Ols03} studied ``extreme non-normal'' numbers
and continued fractions which are, in particular, wild
historic points, and so form a topologically generic subset;
see Example~\ref{ex:extremenonnormal} in
Section~\ref{sec:topolog-markov-chain}.

We note that there are classes of systems with no
specification where our construction can be performed,
e.g. \emph{Lorenz-like or singular-hyperbolic attractors},
and other classes of non-uniformly hyperbolic and singular
flows as \emph{sectional-hyperbolic flow}; see
Examples~\ref{ex:hyperb-measur-Lorenz}
and~\ref{ex:hyp-meas-sec-hyp} and also Remark
\ref{rmk:nospecification} in
Subsection~\ref{sec:hyperb-measur-diffeo}. Hence, the
genericity of wild historic points is more general than the
genericity of points with maximal oscillation in systems
with specification.

\section{Statement of results}
\label{sec:statement-results}

We need some preliminary definitions to be able to state the results.

\subsection{Measure associated to a sequence of outer
  measures}
\label{sec:measure-associ-seque}

Let $\XX$ be a compact metric space, $f:\XX\to\XX$ a
measurable map with respect to the Borel $\sigma$-algebra $\fB$
and let $d$ be the distance on $\XX$.  Let
$(\xi_j)_{j\in\NN}$ be a sequence of finitely additive outer
probabilities on $\XX$, so that each $\xi_j$ is a set
function defined on all parts of $\XX$ with values in
$[0,1]$ and such that $\xi_{j}(\XX)=1$ $\forall\,j$.

\begin{definition}\label{def:premeasure}
  Given a sequence $(\xi_j)_{j\in\NN}$ of finitely additive
  outer measures with unit total mass, we define
  \begin{align*}
    \tau(A)=\limsup_{n\to\infty}
    \frac{1}{n}\sum_{j=0}^{n-1}\xi_j(A).
  \end{align*}
\end{definition}
It is easy to see that
$$
\tau(A)\le\tau(A\cup B)\le\tau(A)+\tau(B)
$$
for all $A,B\subset\XX$ and consequently
\begin{equation}\label{eq9876567}
\tau\bigg(\bigcup_{j=1}^{n}A_{j}\bigg)\le\sum_{j=1}^{n}\tau(A_{j}),
\end{equation}
for any finite collection $A_{1},\cdots,A_{n}$ of subsets of
$\XX$. 

However, one can easily find examples of countable
collections $(A_{j})_{j\in\NN}$ of sets and of measures
$(\xi_j)_{j\in\NN}$ such that
\begin{align*}
 \tau\bigg( \bigcup_{j\in\NN} A_{j} \bigg) >
 \sum_{j\in\NN}\tau(A_{j}).
\end{align*}
\begin{example}\label{ex:notsummable}
  If the positive $f$-orbit
  $\cO^{+}(x)=\{x_j=f^j(x):j\in\NN\}$ is infinite then,
  taking $A_{j}=\{x_{j}\}$ and $\xi_j$ the Dirac point mass
  at $x_j$, i.e. $\xi_j=\delta_{x_j}$, we get
  $\tau(A_{j})=0$ for each $j\ge1$ and so
  $1=\tau(\cO^{+}(x))=\tau(\cup
  A_{j})>\sum_{j}\tau(A_{j})=0$.
\end{example}

Let $\cA$ be the collection of open sets of $\XX$.
According to the definition of pre-measure of
Rogers~\cite{rogers}, $\tau$ restricted to $\cA$ is a
\emph{pre-measure}; see \cite[ Definition~5]{rogers}. 

\begin{definition}\label{def:charmeasureII}
  Given $Y\subset\XX$, define
  \begin{align*}
    \nu(Y)=\sup_{r>0}\nu_{r}(Y)\,\,\big(=\lim_{r\searrow0}\nu_{r}(Y)\big),
  \end{align*}
  where
  $\nu_{r}(Y)=\inf_{_{\cI\in\cA(r,Y)}}\sum_{I\in\cI}\tau(I)$
  and $\cA (r,Y)$ is the set of all countable covers
  $\cI=\{I_i\}$ of $Y$ by elements of $\cA $ with
  $\diam(I_i)\le r$ $\forall\,i$.
\end{definition}
The function $\nu$, defined
on the class of all subset of $\XX$, is called in \cite{rogers}
the {\em (Caratheodory) metric measure constructed from the
  pre-measured $\tau$ by Method~II} (Theorem~15 of
\cite{rogers}). 

We define $\eta$ to be the restriction of $\nu$ to the Borel
sets, i.e., $\eta=\nu\mid\fB$. From \cite[Theorem~23]{rogers}
(see also \cite[Theorem~3]{rogers}) it follows that $\eta$ is a
countable additive measure defined on the $\sigma$-algebra
$\fB$ of Borel sets.

Certain measure-theoretical properties of
$\eta$ and those of $\tau$ are provided in
Section~\ref{sec:additive-measure-eta}: first with no
dynamical assumptions, and then assuming that $\tau$ is
$f$-invariant.

\subsection{Wild historic behavior generically}
\label{sec:wild-histor-behavi}

Considering the particular choice $\xi_j=\delta_{x_j}$ with
$x_j=f^j(x_0), j\ge0$ for a given $x_0\in\XX$, it is easy to
see that the associated set function $\tau$ satisfies
$\tau(f^{-1}(A))=\tau(A)$ for every subset $A$ of $\XX$. We
denote by $\eta_{x}=\eta$ the measure obtained from $\tau$
with the above choices.

\begin{example}
  \label{ex:Boweneye}
  We consider the well-known example of a planar flow with
  divergent time averages attributed to Bowen; see
  Figure~\ref{fig.bowen} and \cite{Ta95}. We set $f=X_1$ the
  time-$1$ map of the flow, and define $\xi_j=\delta_{x_j}$
  with $x_j=f^j(x_0)$ for $j\ge1$, where $x_0$ is a point in
  the interior of the plane curve formed by the heteroclinic
  orbits connecting the fixed hyperbolic saddle points
  $A,B$.

\begin{figure}[h]
\begin{center}
\includegraphics[height=1in]{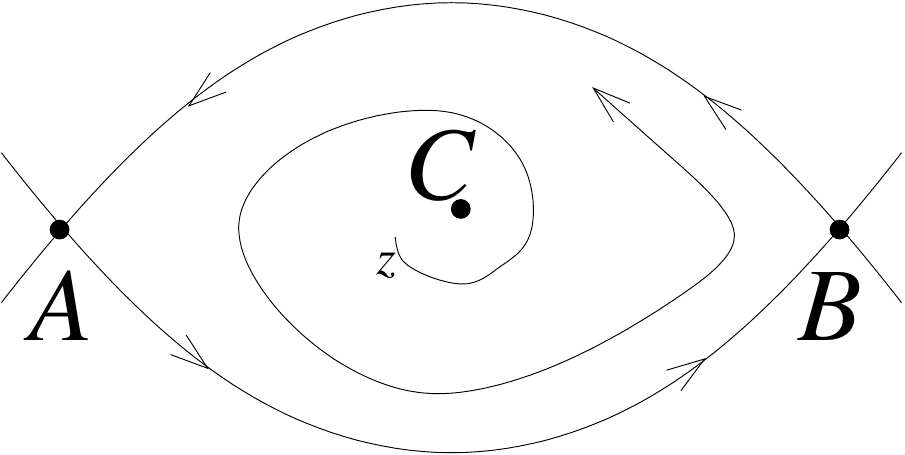}
\caption{\label{fig.bowen} A planar flow with divergent time
averages}
\end{center}
\end{figure}
We assume that this cycle is attracting, that is, the
homoclinic connections are also the set of accumulation
points of the positive orbit of $x_0$; in particular, we
have $|\det Df|<1$ near the stable/unstable manifold of the
saddles $A,B$. However, this
accumulation is highly unbalanced statistically. 

Indeed, if we denote the expanding and contracting
eigenvalues of the linearized vector field at $A$ by
$\alpha_+$ and $\alpha_-$ and at $B$ by $\beta_+$ and
$\beta_-$, and the modulus associated to the upper and lower
saddle connection by
  \begin{align*}
    \lambda=\alpha_-/\beta_+ \qand \sigma=\beta_-/\alpha_+
  \end{align*}
  then $\lambda>0$, $\sigma>0$ and $\lambda\sigma>1$, since
  the cycle is assumed to be attracting; see \cite{Ta95},
  where it is proved that for every continuous function
  $\vfi:\RR^2\to\RR$ with $\vfi(A)>\vfi(B)$ we have
  \begin{align*}
    \limsup_{n\to+\infty}\frac1n\sum_{j=0}^{n-1}\vfi(x_j)
    &=
    \frac{\sigma}{1+\sigma}\vfi(A)+\frac1{1+\sigma}\vfi(B), \qand
    \\
    \liminf_{n\to+\infty}\frac1n\sum_{j=0}^{n-1}\vfi(x_j)
    &=
    \frac{\lambda}{1+\lambda}\vfi(B)+\frac1{1+\lambda}\vfi(A).
  \end{align*}
  In particular, for any small $\epsilon>0$
  \begin{align*}
    \limsup_{n\to+\infty}\frac1n\sum_{j=0}^{n-1}\chi_{B_\epsilon(A)}(x_j)
    \ge\frac{\sigma}{1+\sigma}
    \qand
    \limsup_{n\to+\infty}\frac1n\sum_{j=0}^{n-1}\chi_{B_\epsilon(B)}(x_j)
    \ge\frac{\lambda}{1+\lambda},
  \end{align*}
  so that $\tau(B_\epsilon(A))\ge\sigma(1+\sigma)^{-1}$ and
  $\tau(B_\epsilon(B))\ge\lambda(1+\lambda)^{-1}$ for all
  $\epsilon>0$. Since we can find $\epsilon>0$ arbitrarily
  small such that $B_\epsilon(A),B_\epsilon(B)$ are in
  $\cG$, then we deduce 
  \begin{align*}
    2>\eta_{x_0}(\XX)
    =
    \eta_{x_0}(A)+\eta_{x_0}(B)
    =
    \frac{\sigma}{1+\sigma}+\frac{\lambda}{1+\lambda}
    =
    \frac{\sigma+\lambda+2\lambda\sigma}{1+\sigma+\lambda+\lambda\sigma}
    >1
  \end{align*}
  because $\lambda\sigma>1$.
\end{example}

\begin{remark}\label{rmk:tau0}
  Since orbits of all points $x$ except the fixed point $C$ accumulate
  one of the saddle fixed points $A,B$ and/or the heteroclinic
  connections $W^u(A)\cup W^u(B)$, we have that
  $\tau_x(B_\epsilon(K))=0$ for each compact subset $K$ in the
  interior of the domain bounded by the heteroclinic connection and
  all small enough $\epsilon>0$. Moreover, $\tau_x(B_\epsilon(L))=0$
  for all small enough $\epsilon>0$ and each compact subset $L$ of
  $(W^u(A)\cup W^u(B))\setminus\{A,B\}$ the stable/unstable sets of
  $A,B$ excluding $A,B$.
\end{remark}

\begin{definition}
  \label{def:historic}
  We say that $x\in\XX$ for which $\eta_x$ is not a
  probability measure is a \emph{historic point} or a point
  with \emph{historic} behavior. We denote the set of
  historic points of the map $f$ by $\cH=\cH_f$.
\end{definition}

This definition follows the one in Takens~\cite{takens08}.  The
existence of such points can be obtained in any compact invariant set
of a map which is conjugate to a full shift, e.g. an horseshoe. It can
be easily adapted to any topological Markov chain (Markov subshift of
finite type) and hence can be applied to any basic set of an Axiom A
diffeomorphism; see e.g. \cite{Sm67,BR75,Bo75} for the definitions of
Axiom A diffeomorphisms and \cite{takens08}.

In fact, Dowker~\cite{Dowk53} showed that if a transitive
homeomorphism of a compact set $\XX$ admits a point $z$ with
transitive orbit and historic behavior, then historic points
form a topologically generic subset of $\XX$. As observed
in~\cite{takens08}, the generic subset of points with
historic behavior exists also in the stable set of any basic
set $\Lambda$ of an Axiom A diffeomorphism.

Combining the construction in \cite{takens08} with density of
hyperbolic periodic points we can obtain a stronger property for a
topologically generic subset points.

\begin{definition}
  \label{def:wild}
  We say that $x\in\XX$ for which $\eta_x$ gives infinite
  mass to every open subset of a compact invariant subset
  $\Lambda$ for the dynamics, is a point with \emph{wild
    historic behavior in $\Lambda$} or a \emph{wild
    historic point}. We denote the set of wild historic
  points of the map $f$ by $\W=\W_f$.
\end{definition}

\begin{remark}
  \label{rmk:finvtaueta}
  It is clear by definition that
  $\tau_x^f=\tau_{fx}^f=\tau_y^f$ for each $x,y\in\XX$ so
  that $fy=x$. It follows that
  $\eta_x=\eta_{fx}=\eta_y$. Hence, the set $\cH_f$ of
  historic points and the set $\W_f$ of wild historic points
  are both $f$-invariant: $f^{-1}(\cH_f)=\cH_f$ and
  $f^{-1}(\W_f)=\W_f$. It is also clear that
  $\W_f\subset\cH_f$.
\end{remark}

The following result provides plenty of classes of examples
of abundance of wild historic behavior.

  \begin{maintheorem}
    \label{mthm:genericallywild}
    The set of points with wild historic behavior in 
    \begin{enumerate}
    \item every mixing topological Markov chain with a
      denumerable set of symbols (either one-sided or
      two-sided);
    \item every open continuous transitive and positively
      expansive map of a compact metric space;
    \item each local homeomorphism defined on an open dense
      subset of a compact space admitting an induced full
      branch Markov map;
    \item suspension semiflows, with bounded roof functions,
      over the local homeomorphisms of the previous item;
    \item the stable set of any basic set $\Lambda$ of
      either an Axiom A diffeomorphism, or a Axiom A vector
      field;
    \item the support of an expanding measure for a
      $C^{1+}$ local diffeomorphism away from a
      non-flat critical/singular set on a compact manifold;
    \item the support of a non-atomic hyperbolic measure for
      a $C^{1+}$ diffeomorphism, or a $C^{1+}$ vector field,
      of a compact manifold;
    \end{enumerate}
    is a topologically generic subset (denumerable
    intersection of open and dense subsets).
  \end{maintheorem}

  We stress that item (2) provides mild conditions to obtain
  a generic subset of wild historic points: positively
  expansive maps $f$ are those for which there exists
  $\delta>0$ so that any pair $x,y$ of distinct points are
  guaranteed to be $\delta$-apart in some finite time, that
  is, there exists $N=N(x,y)\ge0$ so that
  $d(f^Nx,f^Ny)>\delta$ (a variation of ``sensitive
  dependence''). This property together with the existence
  of a dense orbit in a compact metric space ensures the
  existence of plenty wild historic points.

  For detailed definitions of positively expansive, induced
  full branch Markov map, expanding measure, Axiom A systems
  and hyperbolic measures, see
  Section~\ref{sec:wild-histor-behavi-1} and references
  therein.

  The connection between properties of real numbers in the
  interval $[0,1]$ and properties of the orbits of the maps
  $T_b:[0,1]\circlearrowleft, x\mapsto bx\bmod1$ for any
  integer $b\ge2$ and the Gauss map
  $G:[0,1]\circlearrowleft, x\mapsto \frac1x\bmod1$,
  together with the natural coding of the dynamics by
  topological Markov chains, enables us to easily relate
  wild historic behavior with absolutely abnormal numbers
  and extremely non-normal continued fractions; see e.g.
  \cite{olsen2004,olsen2004a,olsen-winter07} and references
  therein. Theorem~\ref{mthm:genericallywild} contains some
  of the genericty results in these works as particular
  cases; see e.g. Example~\ref{ex:extremenonnormal} in
  Section~\ref{sec:topolog-markov-chain}.

  We have obtained a topologically generic subset of wild
  historic points for abundant classes of systems with (non
  uniform) hyperbolic behavior. This indicates that
  \emph{the absence of wild historic points implies that all
    invariant probability measures are either atomic or
    non-hyperbolic}, that is, no wild historic behavior
  forces every non-atomic invariant probability measure to
  have zero Lyapunov exponents, either for diffeomorphisms,
  vector fields or endomorphisms of compact manifolds; and
  even for local diffeomorphims away from a singular
  set. The only extra assumptions being sufficient
  smoothness (H\"older-$C^1$ seems to be enough) and a
  singular/critical set regular enough. We present some
  conjectures along these lines in
  Section~\ref{sec:some-open-questions}.

\subsection{Historic behavior and heteroclinic attractors}
\label{sec:open-histor-behavi}

The following result shows that strong historic behavior in
the neighborhood of two hyperbolic periodic points is
sufficient to obtain an heteroclinic connection relating the
two points.

\begin{maintheorem}\label{mthm:openhistoric-heteroclinic}
  Let $f:M\to M$ be a $C^1$ diffeomorphism on a compact
  boundaryless manifold $M$ endowed with a pair
  of hyperbolic periodic points $P,Q$ satisfying, for
  some $\epsilon>0$
\begin{itemize}
\item[(H)] $\eta_x$ is atomic with two atoms $P,Q$ for every
  $x$ either in $B_\epsilon(P)\setminus(W^s_\epsilon(P)\cup
  W^u_\epsilon(P))$ or $ B_\epsilon(Q)\setminus(W^s_\epsilon(Q)\cup
  W^u_\epsilon(Q))$.
\end{itemize}
Then $P$ and $Q$ have a heteroclinic cycle: $W^u(P)=W^s(Q)$
and $W^s(P)=W^u(Q)$.
\end{maintheorem}

That is, this result gives a sufficient condition for a
diffeomorphism to exhibit an heteroclinic attractor, as in
the example of Bowen as presented in \cite{Ta95} and
Example~\ref{ex:Boweneye}.

\subsection{Organization of the text}
\label{sec:organization-text}

We present the construction of the measure $\eta$ through a
pre-measure $\tau$ in a very general dynamical setting in
Section~\ref{sec:additive-measure-eta}, where we also
consider the measure $\eta_x$ for the orbit of a given point
$x$. 
We construct the
generic subset of wild historic points in
Section~\ref{sec:wild-histor-behavi-1} proving
Theorem~\ref{mthm:genericallywild} and providing abundant
classes of examples to apply our results. 
Finally, in Section~\ref{sec:strong-histor-behavi} we prove
Theorem~\ref{mthm:openhistoric-heteroclinic} and in
Section~\ref{sec:some-open-questions} we propose some
conjectures.

\subsection*{Acknowledgments}

We thank 
Varandas (UFBA) for very useful comments and references.

Both authors were partially supported by CNPq (V.A. by
projects 477152/2012-0 and 301392/2015-3; V.P. by project
304897/2015-9), FAPESB and PRONEX-Dynamical Systems
(Brazil). V.P. would also like to thank the finantial
support from Balzan Research Project of J. Palis.


\section{The measure $\eta$ versus the pre-measure $\tau$}
\label{sec:additive-measure-eta}

Here we provide some properties of the pre-measure $\tau$ and of the
measure $\eta$ without assuming invariance of $\tau$; where $\tau$ and
$\eta$ are given according to Definitions~\ref{def:premeasure}
and~\ref{def:charmeasureII}, respectively.

\subsection{Some general properties of $\eta$}
\label{sec:some-general-propert}

Here we prove some properties of $\eta$ and $\tau$ of a general
nature, without special dynamical assumptions.

\begin{lemma}\label{le:uncountable-K}
  If $\cK$ is a uncountable collection of pairwise disjoint
  compact sets of $\XX$, then there is a countable
  sub-collection $\cK_{0}\subset\cK$ such that
  $\lim_{\varepsilon\to0}\tau(B_{\varepsilon}(K))=0$ for
  every $K\in\cK\setminus\cK_{0}$.
\end{lemma}

\begin{proof}
  Let $\cK_{n}$ be the family of $K\in\cK$ such that
  $\frac{1}{n}<\lim_{\varepsilon\to0}\tau(B_{\varepsilon}(K))
  \le\frac{1}{n-1}$, for $n>1$, and let
  $\cK_{0}=\bigcup_{n\ge1}\cK_{n}$.  If $\cK_{0}$ is
  uncountable, then there is some $n\ge1$ such that
  $\cK_{n}$ is uncountable.  As
  $\varepsilon\mapsto\tau(B_{\varepsilon}(K))$ is an
  increasing function, it follows that
  $\tau(B_{\varepsilon}(K))>1/n$ for all $K\in\cK_{n}$ and
  all $\varepsilon>0$.  Set $\cK^{1}=\cK_{n}$ and let
  $\cK^{1}_{j}$ be the collection of $K\in\cK^{1}$ such that
  $\frac{1}{j}\sum_{i=0}^{j-1}\xi_i(B_{\varepsilon}(K))>\frac{1}{2n}$
  for each $\varepsilon>0$.

  As $\cK^{1}$ is uncountable, there are infinitely many
  $j\in\NN$ such that $\cK^{1}_{j}$ is uncountable.  Thus,
  let ${s}$ be such that $\cK^{2}=\cK^{1}_{{s}}$ is
  uncountable. 

  Let $K_{1},K_{2},\cdots,K_{4n}\in\cK^{2}$ be any finite
  collection of elements of $\cK^{2}$ with $K_{i}\ne K_{t}$
  for $i\ne t$. Let also $\varepsilon>0$ be such that
  $B_{\varepsilon}(K_{i})\cap
  B_{\varepsilon}(K_{t})=\emptyset$ for $i\ne t$. For
  $L=B_{\varepsilon}(K_{1})\cup\cdots\cup B_{\varepsilon}(K_{4n})$
  we have 
  \begin{align*}
    1 \ge \frac{1}{s}\sum_{i=0}^{s-1}\xi_i(L) =
    \frac{1}{s}\sum_{t=0}^{4n-1}\sum_{i=0}^{{s}-1}
    \xi_i(B_{\varepsilon}(K_{t})) > 4n\frac{1}{2n} =2
  \end{align*}
  a contradiction. Hence $\cK_{0}$ is countable and, as
  $\lim_{\varepsilon\to0}\tau(B_{\varepsilon}(K))=0$ for
  every $K\in\cK\setminus\cK_{0}$ the proof is complete.
\end{proof}

The previous result allows us to show that compact sets with null
$\eta$-measure are those compacta with neighborhoods of vanishing
pre-measure.

\begin{lemma}\label{Lema005454a}
  If $K\subset\XX$ is a compact set, then $\eta(K)=0$ $\iff$
  $\lim_{\varepsilon\to0}\tau(B_{\varepsilon}(K))=0$.
\end{lemma}

\dem
First suppose that $\eta(K)=0$. Given $\delta>0$ let
$\cc$ be an open cover of $K$ such that
$\sum_{A\in\cc}\tau(A)<\delta$. As $K$ is compact there is a
finite subcover $\cc'\subset\cc$ of $K$. Clearly
$\sum_{A\in\cc'}\tau(A)\le\sum_{A\in\cc}\tau(A)<\delta$. Let
$\varepsilon_{0}>0$ be such that
$B_{\varepsilon_{0}}(K)\subset\bigcup_{A\in\cc'}A$. As
$\cc'$ is finite, we get
$\tau(B_{\varepsilon}(K))\le\tau(\bigcup_{A\in\cc'}A)\le\sum_{A\in\cc'}\tau(A)<\delta$
$\forall0<\varepsilon\le\varepsilon_{0}$. As a consequence,
$\lim_{\varepsilon\to0}\tau(B_{\varepsilon}(K))=0$.

Now, suppose that
$\lim_{\varepsilon\to0}\tau(B_{\varepsilon}(K))=0$. Let
$\cc\in\cA(r,K)$ be a finite cover of $K$.  Given $\delta>0$, let
$\varepsilon>0$ be such that
$\tau(B_{\varepsilon}(K)<\frac{\delta}{\#\cc}$. Let
$\cc'=\{A\cap B_{\varepsilon}(K)$ $;$ $A\in\cc\}$. Note that
$\cc'\in\cA(r,K)$ and
$\sum_{A\in\cc'}\tau(A)=\sum_{A\in\cc}\tau(A\cap
B_{\varepsilon}(K))\le\sum_{A\in\cc}\tau(B_{\varepsilon}(K))<\#\cc\frac{\delta}{\#\cc}=\delta$. Thus,
$\nu_{r}(K)=0$ $\forall\,r>0$ and, as a consequence,
$\eta(K)=0$.
\cqd

As a direct consequence of Lemmas~\ref{le:uncountable-K}
and~\ref{Lema005454a} we get the following.

\begin{corollary}\label{Cor005454a}
  If $\cK$ is a uncountable collection of pairwise disjoint
  compact subsets of $\XX$, then there is a countable
  sub-collection $\cK_{0}\subset\cK$ such that $\eta(K)=0$ for
  every $K\in\cK\setminus\cK_{0}$.
\end{corollary}

Now we show that the $\eta$-measure dominates the pre-measure of the
closure of open sets with negligible boundary.

\begin{lemma}\label{Lema005454b}
  If $A\subset\XX$ is an open set with $\eta(\partial A)=0$
  then $\tau(\overline{A})\le\eta(A)$.
\end{lemma}
\dem Let $\delta>0$ be a positive number.  As $\eta$ is
regular, let $\varepsilon_{0}>0$ be small so that
$\eta(K_{\varepsilon_{0}})>\eta(A)-\delta/2$, where
$K_{\varepsilon}=A\setminus B_{\varepsilon}(\partial A)$. By
lemma~\ref{Lema005454a}, we can take
$0<\varepsilon<\varepsilon_{0}/2$ so that
$\tau(B_{2\varepsilon}(\partial A))<\delta/2$.

Let $r>0$ be small enough so that
$\nu_{r}(A)\le\eta(A)\le\nu_{r}(A)+\delta/2$. Let
$\cc\in\cA(r,A)$ be such that
$\nu_{r}(A)\le\sum_{V\in\cc}\tau(V)\le\nu_{r}(A)+\delta/2$. Thus,
$\eta(A)\le\sum_{V\in\cc}\tau(V)+\delta/2\le\eta(A)+\delta$,
or equivalently,
$$\bigg|\eta(A)-\sum_{V\in\cc}\tau(V)\bigg|\le\delta/2.$$
As $\cc_{1}=\{V\cap A\,;\,V\in\cc\}\in\cA(r,V)$ and
$\sum_{V\in\cc_{1}}\tau(V)\le\sum_{V\in\cc}\tau(V)$,
changing $\cc$ by $\cc_{1}$ if necessary, we may assume that
$V\subset A$ $\forall V\in\cA(r,A)$.

Now since $K_{\varepsilon}$ is compact, there is some finite
open cover of $\cc'\subset\cc$ of
$K_{\varepsilon}$. Therefore,
$A':=\bigcup_{V\in\cc'}V\subset A$ and
$\sum_{V\in\cc'}\tau(V)\le\sum_{V\in\cc}\tau(V)<\eta(A)+\delta/2$. As
$\cc'$ is finite, we get $\tau(A')\le\sum_{V\in\cc'}\tau(V)$
and so,
$$\tau(A')\le\eta(A)+\delta/2.$$
As $\overline{A}\subset A'\cup B_{2\varepsilon}(\partial
A)$, we get $\tau(\overline{A})\le\tau(A')+\delta/2$. As a
consequence $\tau(\overline{A})\le\eta(A)+\delta$ for
every $\delta>0$.  \cqd

In what follows we define $\cG_{0}$ to be the subfamily of
open subsets $U\subset\XX$ such that $\eta(\partial U)=0$;
and then define
\begin{align*}
  \cG=\{U\in\cG_{0} : \eta(\partial f^{-1}(U))=0\}.
\end{align*}
This family is enough to study the topology of $\XX$.

\begin{lemma}\label{Lema005454c}
  The family $\cG_{0}$ generates the topology of $\XX$ (and
  also the Borel sets). Furthermore, if $f$ is continuous,
  then $\cG$ generates the topology of $\XX$.
\end{lemma}

\dem It is enough to show that for any given $x\in\XX$ and
any $r>0$, with $\partial B_{r}(x)\ne\emptyset$, there is a
sequence $r_{n}\nearrow r$ such that $\eta(\partial
(B_{r_{n}}(x)))=0$. 

We consider $T_{r,\epsilon}=\{s\in(r-\epsilon,r): \eta(\partial
B_s(x))>0\}$ for any fixed $0<\epsilon<r$. Then
$\cK_\epsilon:=\{\partial B_s(x) : s\in T_{r,\epsilon}\}$ is a
pairwise disjoint collection of compact subsets of $\XX$
with positive measure. From Corollary~\ref{Cor005454a} this
collection must be countable. Hence there exists
$s\in(r-\epsilon,r)$ such that $\eta(\partial
B_s(x))=0$. Since $\epsilon,r$ can be taken arbitrarily
close to zero, this proves what we need.

For a continuous $f$ we replace $B_s(x)$ by $f^{-1}(B_s(x))$
in the definition of $T_{r,\epsilon}$ and note that
$\partial \big(f^{-1}(B_s(x))\big)$ is compact since
$f^{-1}(B_s(x))$ is open. The same argument above applies in
this case and completes the proof.
\cqd

The pre-measure coincides with $\eta$ on open sets if $\eta$ is a
probability measure.

\begin{lemma}
  If $\eta$ is a probability, then
  $\eta(A)=\tau(A)$ for each $A\in\cG_{0}$.
\end{lemma}

\dem We assume that $\eta$ is a probability. As
$A\in\cG_{0}\iff\XX\setminus\overline{A}\in\cG_{0}$, it
follows from lemma~\ref{Lema005454b} that $\tau(\XX\setminus
A) = \tau(\overline{\XX\setminus\overline{A}}) \le
\eta(\XX\setminus\overline{A}) \le \eta(\XX\setminus A)$ for
every $A\in\cG_{0}$. Hence $\tau(A)=
\tau(\XX\setminus(\XX\setminus A))
\le\eta(\XX\setminus(\XX\setminus A))=\eta(A)$.

Now we assume, by contradiction, that there is some
$A\in\cG_{0}$ such that $\tau(A)<\eta(A)$. As
$\tau(\XX\setminus A)\le\eta(\XX\setminus A)$, we get
$1=\tau(\XX)\le\tau(A)+\tau(\XX\setminus
A)<\eta(A)+\eta(\XX\setminus A)=\eta(\XX)=1$, a
contradiction, completing the proof. \cqd


Consequently, $\eta$ dominates the pre-measure of compacta.

\begin{corollary}\label{Cor003463}
  If $\eta$ is a probability, then $\tau(K)\le\eta(K)$ for every
  closed set $K\subset\XX$.
\end{corollary}
\dem Choose $\varepsilon_{n}\to0$ so that $\eta(\partial
B_{\varepsilon_{n}}(K))=0$ for each $n\in\NN$. Thus
$\tau(K)\le\tau(B_{\varepsilon_{n}}(K))=\eta(B_{\varepsilon_{n}}(K))\to\eta(K)$.
\cqd

\begin{definition}\label{def:integra}
  Given $(\xi_j)_{j\in\NN}$ as before and $\varphi:\XX\to\RR$ a
  bounded measurable function, we set
  \begin{align*}
    \tau(\varphi):=
    \limsup_{n\to+\infty}
    \frac1n\sum_{j=0}^{n-1}\xi_i(\varphi)
    =
    \limsup_{n\to+\infty}\frac1n
    \sum_{j=0}^{n-1}\int\varphi \,d\xi_j.
  \end{align*}
\end{definition}
Clearly $\tau(\varphi+\psi)\le\tau(\varphi)+\tau(\psi)$ and so,
$\sum_{j=1}^{k}\tau(\phi_{n})\ge\tau(\sum_{j=1}^{k}\phi_{n})$ for any
finite collection of functions
$\varphi,\psi,\phi_1,\dots,\phi_k:\XX\to\RR$. We need the following
result which can be found in any standard textbook on measure and
integration; see e.g.~\cite{Ru21}.

\begin{lemma}\label{Lema05633}
  Let $\mu$ be a finite regular measure. If
  $\varphi:\XX\to\RR$ is a continuous function, then there
  exists a sequence of simple function
  $\varphi_{n}=\sum_{j=1}^{\ell_{n}}a_{j,n}\chi_{A_{j,n}}$
  such that $A_{j,n}\in\cG_{0}$, the $(A_{j,n})_j$ are
  pairwise disjoint for each $n$,
  $\varphi_{n}(x)\searrow\varphi$ for $\mu$-almost all
  $x\in\XX$ and also $\int\varphi_{n}d\mu\searrow\int\varphi
  d\mu$.
\end{lemma}

Next we show that $\tau(\vfi)$ is the $\eta$-integral of $\vfi$ is
$\eta$ is a probability.

\begin{lemma}\label{Lema0562572}
  If $\eta$ is a probability, then $\int\varphi\,
  d\eta=\tau(\varphi)$ for each continuous function
  $\varphi:\XX\to\RR$.
\end{lemma}

In particular, we obtain that $\tau=\eta$ is a Borel
probability measure.

\begin{proof}[Proof of Lemma~\ref{Lema0562572}]
  First we show that $\int\varphi
  d\eta\ge\tau(\varphi)$ for every continuous
  function $\varphi$. 

  Indeed, if $\varphi:\XX\to\RR$ is a continuous function,
  let $\varphi_{n}=\sum_{j}a_{j,n}\chi_{A_{j,n}}$ be given
  by Lemma~\ref{Lema05633}. Let $a=\inf\varphi$,
  $\psi=\varphi-a$ and
  $\psi_{n}=\sum_{j}(a_{j,n}-a)\chi_{A_{j,n}}$. Note that
  $\psi, \psi_{n}$ and $(a_{j,n}-a)\chi_{A_{j,n}}$ are
  nonnegative functions and $\psi_{n}\searrow\psi$. Since
  $A_{j,n}\in\cG_{0}$ we get $\eta(A_{j,n})=\tau(A_{j,n})$
  $\forall\,j,n$. Thus, using that $(a_{j,n}-a)\ge0$ we
  have
  \begin{align*}
    \int\psi_{n}\,d\eta
    &=
    \sum_{j=1}^{\ell_{n}}(a_{j,n}-a)\eta(A_{j,n})
    =
    \sum_{j=1}^{\ell_{n}}(a_{j,n}-a)\tau(A_{j,n})
    \\
    &\ge\tau\bigg(\sum_{j=1}^{\ell_{n}}(a_{j,n}-a)\chi_{A_{j,n}}\bigg)
    =\tau(\psi_{n}).
\end{align*}
As $\psi_{n}\ge\psi$ implies that
$\tau(\psi_{n})\ge\tau(\psi)$ and since $\int\psi_{n}
d\eta\to\int\psi d\eta$, we get
\begin{align*}
  \int\varphi \, d\eta - a
  =\int\psi \, d\eta
  \ge
  \tau(\psi)
  =
  \tau(\varphi)-a.
\end{align*}
That is, $\int\varphi d\eta\ge\tau(\varphi)$ for every
continuous function $\varphi$.

Now, suppose that $\int\varphi d\eta>\tau(\varphi)$ for some
continuous function $\varphi$. As $1-\varphi$ is also a
continuous function
$\int(1-\varphi)\,d\eta\ge\tau(1-\varphi)$.  Therefore
\begin{align*}
  1
  &=
  \int1 \, d\eta
  =
  \int((1-\varphi)+\varphi)\,d\eta
  =
  \int(1-\varphi)\,d\eta+\int\varphi\,d\eta
  \\
  &>
  \tau(1-\varphi)+\tau(\varphi)
  \ge
  \tau((1-\varphi)+\varphi)=\tau(1)=1.
\end{align*}
This contradiction completes the proof.
\end{proof}

We note that up to this point we have not used any
invariance relation for the measures $\tau$ or $\eta$.

\subsection{The invariant case}
\label{sec:invari-case}

Now we assume that $\tau(f^{-1}(A))=\tau(A)$ for all Borel
subsets $A$ of $\XX$. This depends on the choice of the
sequence $(\xi_i)_{i\in\NN}$ and must be checked for each
specific case.

\begin{lemma}\label{le:tauetainv}
  Let $\tau$ be $f$-invariant: $\tau\circ f^{-1}=\tau$.
  If 
  $\eta$ is a probability measure, then $\eta$ is $f$-invariant.
\end{lemma}

\begin{proof}
  By assumption, we have both $\tau\circ f^{-1}=\tau$ and
  $\eta=\tau$ on the Borel $\sigma$-algebra, thus $\eta$ is
  $f$-invariant.
\end{proof}


\begin{corollary}\label{corollary0562572}
  The following properties are equivalent:
\begin{enumerate}
\item $\eta$ is a probability;
\item
  $\lim_{n\to\infty}\frac{1}{n}\sum_{j=0}^{n-1}\xi_j(\varphi)=\int\varphi
  \,d\eta$ for each continuous function $\vfi$;
\item
  $\lim_{n\to\infty}\frac{1}{n}\sum_{j=0}^{n-1}\xi_j(\varphi)$ exists for each continuous function $\varphi$.
\end{enumerate}
\end{corollary}

\begin{proof}
  First we assume that $\eta$ is a probability and let
  $\varphi\in C(\XX,\RR)$. By Lemma~\ref{Lema0562572}
  applied to $-\varphi$ we get
  \begin{align*}
    \int\varphi \,d\eta
    &=
    -\bigg(\int-\varphi\,d\eta\bigg)
    =
    -\bigg(\limsup_{n\to\infty}\frac{1}{n}\sum_{j=0}^{n-1}\xi_j(-\varphi)
    \bigg)
    \\
    &=
    -\bigg(-\liminf_{n\to\infty}\frac{1}{n}\sum_{j=0}^{n-1}\xi_j(\varphi)
    \bigg)
    =
    \liminf_{n\to\infty}\frac{1}{n}\sum_{j=0}^{n-1}\xi_j(\varphi).
\end{align*}
On the other hand, applying Lemma~\ref{Lema0562572} to
$\varphi$, we obtain
\begin{align*}
  \int\varphi d\eta = \tau(\vfi)=
  \limsup_{n\to\infty}\frac{1}{n}\sum_{j=0}^{n-1}\xi_j(\varphi).
\end{align*}
This is enough to see that item (2) is true if item (1) is
true. Clearly (2) implies (3).

Now we assume that (3) is true and argue by contradiction,
assuming that $\eta(\XX)>1$. Let $r>0$ be small enough so
that $\nu_{r}(\XX)>1$ and
$\cP=\{P_{1},\cdots,P_{s}\}\subset\cG_{0}$ be a finite
collection of disjoint open sets of $\cG_{0}$ with diameters
smaller than $r/2$ and such that
$\XX=\bigcup_{j=1}^{s}\overline{P_{j}}$. Note that
$\cP_{\varepsilon} =
\{B_{\varepsilon/2}(P_{1}),\cdots,B_{\varepsilon/2}(P_{s})\}$
is a cover of $\XX$ by open sets with diameter smaller than
$\frac{r}{2}+\varepsilon$. As $\eta(\partial P_{j})=0$
$\forall\,j$, it follows from Lemma~\ref{Lema005454a} that
$\tau({P_{j}})
\le
\tau(\overline{P_{j}})
\le
\tau(B_{\varepsilon}(P_{j}))
\le
\tau({P_{j}})+\tau(B_{\varepsilon}(\partial P_{j}))
\searrow\tau({P_{j}})$. That is,
$ \tau({P_{j}})
=
\tau(\overline{P_{j}})
=
\lim_{\varepsilon\to0}\tau(B_{\varepsilon}(P_{j}))$
for all $j$.

Considering $0<\varepsilon<r/2$, we get
$\sum_{j=1}^{s}\tau(B_{\varepsilon}(P_{j}))\ge\nu_{r}(\XX)>1.$
Therefore, we arrive at
$\sum_{j=1}^{s}\tau(P_{j})\ge\nu_{r}(\XX)>1$ when
$\varepsilon\to0$.

Since we are assuming that
$\lim_{n\to\infty}\frac{1}{n}\sum_{j=0}^{n-1}\xi_j(\varphi)$
exists for each continuous function $\varphi$ and as
$0\le\xi_{j}\le1$, it is not difficult to show that
$\lim_{n\to\infty}\frac{1}{n}\sum_{j=0}^{n-1}\xi_j(\chi_{A})$
exists for each $A\in\cG_{0}$. That
is,
\begin{align*}
\tau(A)=\lim_{n\to\infty}\frac{1}{n}\sum_{j=0}^{n-1}\xi_j(A),\,\,\forall
A\in\cG_{0}.
\end{align*}
We then obtain the following contradiction
\begin{align*}
  1
  \ge
  \tau(\cup_{j=1}^{s}P_{j})
  =
  \lim_{n\to\infty}\frac{1}{n}\sum_{j=0}^{n-1}\xi_j(\cup_{j=1}^{s}P_{j})
  =
  \sum_{j=1}^{s}\lim_{n\to\infty}\frac{1}{n}\sum_{j=0}^{n-1}\xi_j(P_{j})
  =
  \sum_{j=0}^{n-1}\tau(P_{j})>1.
\end{align*}
We conclude that $\eta$ is a probability measure (recall
that we always have $\eta(\XX)\ge1$) and the proof is
complete.
\end{proof}

Given $x\in\XX$, recall that we set
$x_{j}=f^{j}(x)$ and define the sequence of measures
$\xi_j:=\delta_{x_j}$ for $j\ge0$. Then $\tau$ satisfies
$\tau(f^{-1}(A))=\tau(A)$ for every subset $A$ of $\XX$. We
denote by $\eta_{x}=\eta$ the measure obtained from $\tau$
with the above choices as described in
Section~\ref{sec:statement-results}.

From the last corollary we obtain a characterization of
historic points.

\begin{corollary}
  \label{cor:historic-noaverge}
  For a continuous map $f$ on a compact metric space $\XX$,
  we have that $x$ is a historic point for $f$ if, and only
  if, there exists a continuous function $\vfi:\XX\to\RR$
  such that the sequence of time averages
  $n^{-1}\sum_{j=0}^{n-1}\vfi(f^jx)$ does not converge when
  $n\to+\infty$.
\end{corollary}

\begin{proof}
  It follows directly from Corollary~\ref{corollary0562572}
  since $x\in\cH_f$ means by definition that $\eta_x$ is not
  a probability.
\end{proof}

\subsection{The continous case}
\label{sec:continous-case}

We now assume that not only $\tau$ is $f$-invariant, but
also that $f:\XX\to\XX$ is a continuous map on a metric
space $\XX$.

\begin{theorem}\label{TheoremContinuous}
  If $f$ is continuous, then $\eta_{x}$ is a non trivial
  invariant measure for every $x\in\XX$.
\end{theorem}

\dem We claim that $\eta_{x}(A)\le\eta_{x}(f^{-1}(A))$ for
all $A\in\cG$.  Let $A\in\cG$. Let $r>0$ and
$\cc\in\cA(r,\overline{A})$. As $\overline{A}$ is compact,
let $\cc'\subset\cc$ bet a finite subcover of
$\overline{A}$. Clearly $\cc'\in\cA(r,\overline{A})$
and $\nu_{r}(\overline{A})\le\sum_{V\in\cc'}\tau_{x}(V) \le
\sum_{V\in\cc}\tau_{x}(V)$. Given any set $V$ and
$\varepsilon>0$, define $$[V]_{\varepsilon}=V\setminus
B_{\varepsilon}(\partial V).$$ As $\overline{A}$ is compact
and $\cc'$ is finite, it is easy to see that there is some
$\varepsilon_{0}>0$ such
that $$\cc'_{\varepsilon}:=\{[V]_{\varepsilon}\,;\,V\in\cc'\}\in\cA(r,\overline{A})$$
for every $0<\varepsilon\le\varepsilon_{0}$.  

Moreover,
$\nu_{r}(\overline{A})\le\sum_{V\in\cc'_{\varepsilon}}\tau_{x}(V)\le\sum_{V\in\cc'}\tau_{x}(V)\le\sum_{V\in\cc}\tau_{x}(V)$,
whenever $0<\varepsilon\le\varepsilon_{0}$. 

It follows from Corollary~\ref{Cor005454a} that we can
choose $0<\varepsilon\le\varepsilon_{0}$ so that
$[V]_{\varepsilon}$ and $f^{-1}([V]_{\varepsilon})\in\cG$
for all $V\in\cc'$.  That is, $V,f^{-1}(V)\in\cG$ $\forall
V\in\cc'_{\varepsilon}$.

We write $\cc'_{\varepsilon}=\{V_{1},\cdots,V_{s}\}$ and set
$P_{1}=V_{1}\cap\overline{A}$,
$P_{2}=(V_{2}\cap\overline{A})\setminus P_{1}$, $\cdots$,
$P_{s}=(V_{s}\cap\overline{A})\setminus(P_{1}\cup\cdots\cup
P_{s-1})$.
Observe that
\begin{enumerate}
\item $P_{j}\cap P_{k}=\emptyset$ if $j\ne k$,
\item $B_{\varepsilon'}(P_{j})\subset V_{j}$ for every
  $0<\varepsilon'<\varepsilon$ and every $j=1,\cdots,s$,
\item $\bigcup_{j=1}^{s}P_{j}=\overline{A}$ and
\item $\eta_{x}(\partial P_{j})=\eta_{x}(\partial
  (f^{-1}(P_{j})))=0$ for all $j=1,\cdots,s$.
\end{enumerate}

Thus, 
$$\nu_{r}(\overline{A})\le\sum_{j=1}^{s}\tau_{x}(B_{\varepsilon'}(P_{j}))\le\sum_{j=1}^{s}\tau_{x}(V_{j}).$$
On the other hand $\eta_{x}(\partial P_{j})=0$, we get
$\lim_{\varepsilon'\to0}\tau_{x}(B_{\varepsilon'}(\partial
P_{j}))=0$ and
so, $$\lim_{\varepsilon'\to0}\tau_{x}(B_{\varepsilon'}(P_{j}))=\tau_{x}(P_{j})\,\,\,\forall\,j=1,\dots,s.$$

So, we can conclude that 
$$\nu_{r}(\overline{A})\le\sum_{j=1}^{s}\tau_{x}(P_{j}).$$
As $\chi_{\YY}(f^{j}(x))=\chi_{f^{-1}(\YY)}(f^{j-1}(x))$,
for every $\YY\subset\XX$, we get that $\tau_{x}\circ
f^{-1}=\tau_{x}$ and so,
$$
\nu_{r}(\overline{A})\le\sum_{j=1}^{s}\tau_{x}(P_{j})=
\sum_{j=1}^{s}\tau_{x}(f^{-1}(P_{j})).
$$
As $\eta_{x}(\partial f^{-1}(P_{j}))=0$, we get
$\tau_{x}(\partial f^{-1}(P_{j}))=0$
(from Lemma~\ref{Lema005454a}). Applying Lemma~\ref{Lema005454b}
to $\interior f^{-1}(P_{j})$, it follows that
$$
\tau_{x}(f^{-1}(P_{j}))\le\eta_{x}(\overline{\interior f^{-1}(P_{j})})\le \eta_{x}(\interior (f^{-1}(P_{j})))\le \eta_{x}(f^{-1}(P_{j})).
$$
Therefore,
$$
\nu_{r}(\overline{A})\le\sum_{j=1}^{s}\tau_{x}(f^{-1}(P_{j}))\le\sum_{j=1}^{s}\eta_{x}(f^{-1}(P_{j}))=\eta_{x}\bigg(\sum_{j=1}^{s}f^{-1}(P_{j})\bigg)=\eta_{x}(f^{-1}(\overline{A})).
$$
Taking $r\to0$, we get
$$
\eta_{x}(A)=\eta_{x}(\overline{A})\le\eta_{x}(f^{-1}(\overline{A}))=\eta_{x}(f^{-1}(A)),
$$
whenever $A\in\cG$.

Note that $A\in\cG$ $\Rightarrow$ $\XX\setminus\overline{A}\in\cG$.
Suppose that there are some $A\in\cG$ such that
$\eta_{x}(A)<\eta_{x}(f^{-1}(A))$. Thus
$\eta_{x}(\overline{A})<\eta_{x}(f^{-1}(\overline{A}))$ and
then we reach a contradiction:
$$
\eta_{x}(\XX)=\eta_{x}(\overline{A})+\underbrace{\eta_{x}(\XX\setminus\overline{A})}_{\le
  f^{-1}(\XX\setminus\overline{A}))}<\eta_{x}(f^{-1}(\overline{A}))+\eta_{x}(f^{-1}(\XX\setminus\overline{A}))=\eta_{x}(\XX).$$
As a consequence,
$\eta_{x}(A)=\eta_{x}(f^{-1}(A))\,\,\,\forall\,A\in\cG.$
Using lemma~\ref{Lema005454c}, we can conclude that
$\eta_{x}(E)=\eta_{x}(f^{-1}(E))$ for every Borel set.  \cqd



\section{Wild historic behavior}
\label{sec:wild-histor-behavi-1}

Here we present a proof of
Theorem~\ref{mthm:genericallywild}.

\subsection{Topological Markov Chain}
\label{sec:topolog-markov-chain}

We start with a countable subshift given by a denumerable set $S$ of
symbols (the alphabet, which may be finite with at least two symbols),
an incidence matrix $A=(a_{i,j})_{i,j\in S}$ with $0$ or $1$ entries
and the set
$\XX=\{\underline x=(x_i)_{i\ge0}:a_{x_{i},x{i+1}}=1, i\ge0\}$ of
admissible sequences. We assume that $A$ is aperiodic: there exists
$N\in\ZZ^+$ such that for any pair $b,c\in S$ there are
$x_1,\dots,x_{N-1}\in S$ satisfying
$a_{b,x_1}=a_{x_1,x_2}=\dots =
a_{x_{n-2},x_{n-1}}=a_{x_{N-1},c}=1$. This is the same as requiring
that the matrix $A^N$ have all entries $a^N_{i,j}$ positive.

\subsubsection{Construction of a wild historic orbit}
\label{sec:constr-wild-histor}

We consider the usual left shift map
$\sigma:\XX\circlearrowleft$ and then $(\XX,\sigma)$ is a
mixing topological Markov chain with denumerable set of
states (or symbols).  The aperiodic condition on $A$
ensures, in particular, that there exists a dense orbit
and a denumerable dense subset $\per(\sigma)$ of periodic
orbits; see e.g. \cite{Bo75,Man87}. More precisely, for any
given admissible sequence $a_0,a_1,\dots,a_k$ of symbols and
$q>k+N$ there exists a periodic orbit
$\underline p=(x_k)_{k\ge0}$ such that
$x_0=a_0,\dots,x_k=a_k$ and its (minimum) period
$\pi(\underline p)$ is $q$.
In what follows, we write $\{\underline p_n\}_{n\ge0}$
for an enumeration of a choice of one point of every
periodic orbit in $\per(\sigma)$.
 
We 
consider the following integer sequence defined by
recurrence
  \begin{align*}
    \ell_1=N \qand \ell_{n+1}=10^n\cdot\sum_{i=1}^n \ell_i,
    \quad n>1.
  \end{align*}
  We note that $\ell_{n+1}>\ell_n+N$ 
  for each $n\ge1$.
  We recall the notion of a cylinder set in $\XX$: for a
  given sequence $\underline a=(a_i)_{i\ge1}\in\XX$ and a
  given positive integer $n$ we define
  \begin{align*}
    [\underline a]_n = \{ x=(x_i)_{i\ge1}\in\XX : x_1=a_1,\dots,x_n=a_n\}.
  \end{align*}
  Now we consider the integer sequence $\kappa_n,n\ge1$:
  \begin{align*}
    1, 1, 2 , 1,2,3,1,2,3,4,1,2,3,4,5,1,2,3,4,5,6,\dots
  \end{align*}
  which assumes the value of any given positive integer
  infinitely many times and may be defined through the
  sequence of indexes
  \begin{align*}
    \alpha_0=0\qand\alpha_{n+1}=\alpha_n+n+1, n\ge0
  \end{align*}
  together with the rule $\kappa_{\alpha_n+j}=j$ for all
  $n\ge1$ and $j=1,\dots,n$.  Then we join several strings
  of digits of $\underline p_n$ to form a sequence
  $\underline z\in\XX$ satisfying 
  \begin{align*}
    \sigma^{\ell_n}(\underline z)\in [\underline
    p_{\kappa_n}]_{\ell_n},
    \quad n\ge1.
  \end{align*}
  Such sequence $\underline z$ is admissible (that is, it
  does belong to $\XX$) since $\ell_{n+1}-\ell_n>N$ and so
  the required transitions from the position
  $\ell_n+\kappa_n$ to the position $\ell_{n+1}$ are
  allowed.

  This sequence $\underline z$ is such that its positive
  $\sigma$-orbit visits a very small $\ell_n$-cylinder
  around the $\kappa_n$th element of $\underline p_n$. By
  the definition of $\kappa_n$ and of $\ell_n$, for any
  fixed $\underline p_h$ and any given cylinder $[\underline
  p_h]_m$, with $m\ge1$, we have for all $n\ge1$ such that
  $\kappa_n=h$
  \begin{itemize}
  \item[(P1)] $\sigma^i(\sigma^{\ell_n}(\underline z))$ belongs to
    $[\underline p_h]_m$ for $i=0,\pi_h,2\pi_h,\cdots,
    \tau\pi_h$, where $\tau=[\ell_n/\pi_h]$ and
    $\pi_h=\pi(\underline p_h)$ is the (minimum) period of
    $\underline p_h$;
\item[(P2)] $\sum_{j=\ell_n}^{2\ell_n}
  \chi_{[\underline  p_h]_m}(\sigma^j(\underline z))
  \ge\tau\ge \ell_n\frac{\tau}{\tau\pi_h+r}\ge\frac{\ell_n}{\pi_h}\frac{\tau}{1+\tau}\ge\frac{\ell_n}{2\pi_h}$.
\end{itemize}
Since 
\begin{align*}
  \frac1{2\ell_n+1}\sum_{j=0}^{2\ell_n}
  \chi_{[\underline  p_h]_m}(\sigma^j(\underline z))
  \ge
  \frac1{2\ell_n+1}\sum_{j=\ell_n}^{2\ell_n}
  \chi_{[\underline  p_h]_m}(\sigma^j(\underline z))
  \ge
  \frac{\ell_n}{2\ell_n+1}\cdot\frac1{2\pi_h}
\end{align*}
and
$\kappa_n=h$ for infinitely many positive integers $n$, we
obtain
  \begin{align}\label{eq:cotacilindro}
    \tau_{\underline z}([\underline  p_h]_m)
    =\limsup_{n\to+\infty}
    \frac1n\sum_{j=0}^{n-1}
    \chi_{[\underline  p_h]_m}(\sigma^j(\underline
    z))\ge\frac1{4\pi_h},
\quad\text{for each $h,m\ge1$.}
  \end{align}
  Moreover, if $\underline p_h=(a_i)_{i\ge0}$, then there
  are periodic orbits $p_j\in\per(\sigma)$ whose first $m+1$
  symbols are $a_0,\dots,a_m$ and whose period is any
  positive integer $q_j>m+jN$; and these periodic orbits
  also belong to $[\underline p_h]_m$ (as well as some of
  its iterates, as in item (1) above whenever
  $\kappa_n=j$). In addition, fixing $q_j$, there are at
  least as many periodic orbits $p_j$ with period $q_j$ as
  above as there are letters in $S^j$.\footnote{We denote
    $S^j$ the family of all concatenations of $j$ letters
    from $S$.}  Therefore, (\ref{eq:cotacilindro}) together
  with Lemma~\ref{Lema005454b} implies
  \begin{align}\label{eq:wild}
    \eta_{\underline{z}}([\underline  p_h]_m)
    \ge
    \eta_{\underline{z}}\left(\bigcup_{\underline
        p_j\in\per_{q_j}(\sigma)\cap[\underline p_h]_m}[\underline  p_j]_m\right)
    \ge
    \frac{(\#S)^j}4\cdot\frac1{m+jN}.
  \end{align}
  Even if the set $S$ of symbols is finite, then $(\#S)^j$
  increases exponentially fast in $j$ and we get
  $\eta_z([\underline p_h]_m)=\infty$.  Clearly, we can use
  any $m\ge1$ and also replace $\underline p_h$ by
  $\sigma^i(\underline p_h)$ in the above argument for all
  $i\ge1$. In addition, the set of accumulation points of
  the $\sigma$-orbit of $\underline z$ is $\XX$, because
  $\per(\sigma)$ is dense in $\XX$.

  Since~\eqref{eq:wild} holds for arbitrarily large integers
  $m\ge1$, then
  $\eta_{\un{z}}(\{\sigma^i\un{p}_h\})=+\infty$ for each
  $h\ge1$ and $i\ge1$. Hence $\eta_{\un{z}}$ has countably
  many dense atoms with infinite mass.
  
  Thus, we see that $\eta_{\underline z}(A)=+\infty$ for
  every open subset $A$ of $\XX$, and $z$ is a wild historic
  point.

  \begin{remark}
    \label{rmk:wildperiodic}
    The previous argument show that, in a system with dense
    countable subset of periodic orbits, \emph{if a point
      $z$ satisfies $\tau_zU\ge c(p)>0$ for every
      neighborhood $U$ of a periodic point $p$, where $c(p)$
      depends on $\cO(p)$ only, then $z$ is a wild historic
      point.}
  \end{remark}

  \begin{remark}\label{rmk:notsigmafinite}
    In particular, $A=\{ \un{p}_1, \dots, \un{p}_h\}$ is
    such that $\eta_{\un{z}}(A)=+\infty$ for each given
    $h>1$. So no non-empty subset $B$ of $A$ satisfies
    $\eta_{\un{z}}(B) <\infty$.
  \end{remark}
  
\subsubsection{Genericity of wild historic points}
\label{sec:wild-historic-points}

We claim that such points form a generic subset of
$\XX$. Indeed, let us consider, for any given $n$, the
integer $m_n$ such that, for all $1\le j ,k <n$ with $k\neq
j$ we have $B(\underline p_j,m_n)\cap B(\underline
p_k,m_n)=\emptyset$, where
  \begin{align*}
    B(\underline p_h,m_n) = \bigcup_{1\le i \le
      \pi(\underline p_h)} [\sigma^i(\underline p_h)]_{m_n}
  \end{align*}
  is a neighborhood of the orbit of $\underline p_h$ for any $h\ge1$.
  That is, we take a pairwise disjoint neighborhood of the
  orbit of each one of $\underline p_1,\dots,\underline
  p_{n-1}$, for each $n\ge1$. We observe that clearly
  $m_n\nearrow+\infty$. 

  We then consider the family of sets
  \begin{align*}
    U_n=
    \left\{\underline x\in\XX\mid 
    \forall 1\le j<n \, \exists k=k(j,n)>n :
    \frac1k\sum_{i=0}^{k-1}
    \chi_{B(\underline p_j,m_n)}(\sigma^i(\underline x))
    \ge \frac1{\pi(\un{p}_j)}-\frac1n\right\}
  \end{align*}
  for all integers $n\ge1$. We note that each $U_n$ is open
  in $\XX$, since cylinders are simultaneously closed and
  open sets and $\sigma$ is continuous. In addition,
  $\sigma^i(\underline z)\in U_n$ for all $n,i\ge1$, which
  shows that each $U_n$ is also dense. Thus
  \begin{align}\label{eq:wildY}
  Y=\bigcap_{n\ge1}U_n
  \end{align}
  is a generic subset of $\XX$.  We now show that each
  element $\underline w$ of $Y$ is a wild historic point:
  $\un w\in\W_\sigma$.

  Fix $\underline w\in Y$, a point $\underline a\in\XX$ and $m\ge1$, and
  note that
    \begin{align*}
    \tau_{\underline w}([\underline a]_m)=\limsup_{n\to+\infty}
    \frac1n\sum_{j=0}^{n-1}
    \chi_{[\underline  a]_m}(\sigma^j(\underline w))
    \ge
    \frac1{\pi(\underline p)}
  \end{align*}
  where $\underline p$ is a periodic point $\underline
  p\in\per(\sigma)\cap[\underline a]_m$ of $\sigma$ in the
  open set $[\underline a]_m$. But, as before, we know that
  inside $[\underline a]_m$ we can find periodic points of
  $\sigma$ with all periods $q>m+N$. Hence, as in
  (\ref{eq:wild}) we obtain $\eta_{\underline w}([\underline
  a]_m)=\infty$. Since $\un a\in\XX$ and $m\ge1$ were
  arbitrarily chosen, we conclude that $\un w\in\W_\sigma$.

This completes the proof of the item (1) of
Theorem~\ref{mthm:genericallywild} and shows, in particular,
that any subshift of finite type (i.e., the same as $\XX$
but with a finite alphabet) has a generic subset of wild
historic points.

Moreover, it is clear that the same construction applies
verbatim to two-sided topological Markov chains, for which
the left shift map $\sigma$ is a homeomorphism.

\begin{remark}\label{rmk:specification}
  We note also that a similar construction can be performed
  in any system with specification; see \cite[Chapter
  21]{DGS76}. However, there are classes of systems with no
  specification where our construction can be performed: see
  Example~\ref{ex:hyperb-measur-Lorenz} and Remark
  \ref{rmk:nospecification} in
  Subsection~\ref{sec:hyperb-measur-diffeo}.
\end{remark}

\begin{remark}
  \label{rmk:unitmass}
  By the above construction of $Y$, every $y\in Y$ satisfies
  $\tau_y U=1/\pi(p)$ for every neighborhood $U$ of each
  periodic orbit $p$ in $\XX$. Hence, from
  Remark~\ref{rmk:wildperiodic}, we have that
  \emph{generically $z\in\XX$ is a wild historic point if,
    and only if, $z$ satisfies $\tau_zU\ge c(p)>0$ for every
    neighborhood $U$ of a periodic point $p$, where $c(p)$
    depends on $\cO(p)$ only.}
\end{remark}

\subsubsection{Non-existence of time averages for open and
  dense family of observables}
\label{sec:non-existence-time}

We show that \emph{at a wild historic point time averages do
  not exist for an open and dense subset of continuous
  functions}. Moreover, as shown by Jordan, Naudot and
Young in \cite{JNY09}, \emph{all higher order averages also
  fail to exist}.

We say that $\vfi\in C^0(\XX,\RR)$ is \emph{periodically
  trivial} if $\vfi$ has the same time average over any
periodic orbit, that is
$ \pi(\un q) \sum_{i=0}^{\pi(\un p)-1}\vfi(\sigma^i\un p)
=\pi(\un p)\sum_{i=0}^{\pi(\un q)-1}\vfi(\sigma^i\un q),
\quad\forall \un p,\un q\in\per(\sigma)$.
It is clear that the family of non periodically trivial
functions is an open and dense subset of $C^0(\XX,\RR)$ with
the uniform topology.

\begin{lemma}
  \label{le:nonconst}
  Let $Y\subset\W_\sigma$ be defined as in \eqref{eq:wildY}.
  Given $x\in Y$ and $\vfi\in C^0(\XX,\RR)$, then
  $\frac1n\sum_{j=0}^{n-1}\vfi(f^jx)$ converges when
  $n\to\infty$ if, and only if, $\vfi$ is periodically trivial.
\end{lemma}

\begin{proof}
  Let us assume that
  $\wt\vfi(x)=\lim_{n\to+\infty}\frac1n\sum_{j=0}^{n-1}\vfi(f^jx)$
  exists. By Remark~\ref{rmk:unitmass}, given
  $\un p\in\per(\sigma)$ and any neighborhood $U$ we have
  $\tau_xU=1$. Fixing $\epsilon>0$ we set $U$ so that
  $|\vfi(y)-\vfi(z)|<\epsilon$ for all $y,z\in U$.  Hence we
  can find a sequence $n_k\nearrow\infty$ such that
  $\frac1{n_k}\sum_{j=0}^{n_k-1}\chi_U(\sigma^jx)\ge1-\epsilon$
  for all $k\ge1$ and, consequently, we get
\begin{align*}
    \frac{1}{\pi(\un p)}\left(\sum_{i=0}^{\pi(\un
      p)-1}\vfi(\sigma^i\un p) +\epsilon\right)
      +\epsilon\|\vfi\|_0
  \ge
    \frac1{n_k}\sum_{j=0}^{n_k-1}\vfi(\sigma^jx)
    &\ge
    \frac{1-\epsilon}{\pi(\un p)}\left(\sum_{i=0}^{\pi(\un
      p)-1}\vfi(\sigma^i\un p) -\epsilon\right)
      -\epsilon\|\vfi\|_0.
\end{align*}
Since $\epsilon>0$ is arbitrary, we obtain
$\wt\vfi(x)= \frac{1}{\pi(\un p)}\sum_{i=0}^{\pi(\un
  p)-1}\vfi(\sigma^i\un p)$. Because $\un p\in\per(\sigma)$
is arbitrary, we conclude that $\vfi$ is periodically
trivial.

Reciprocally, let $\un p_1,\un p_2\in\per(\sigma)$ be periodic
orbits over which $\vfi$ has distinct time averages. We
assume without loss of generality that 
  \begin{align*}
    \alpha_1
    :=
    \frac{1}{\pi(\un p_1)}\sum_{i=0}^{\pi(\un p_1)-1}\vfi(\sigma^i\un p_1)
    <
    \frac{1}{\pi(\un p_2)}\sum_{i=0}^{\pi(\un p_2)-1}\vfi(\sigma^i\un p_2)
    =:
    \alpha_2.
  \end{align*}
  Let $U_i$ be neighborhoods of $\cO_\sigma(\un p_i)$ so
  that $\sup\vfi\mid U_i-\inf\vfi\mid U_i<\epsilon, i=1,2$
  and $U_1\cap U_2=\emptyset$. Since $\tau_xU_i=1$, we can
  find sequences $n_k(i)\nearrow\infty$ such that
  $\frac1{n_k(i)}\sum_{j=0}^{n_k(i)-1}\chi_{U_i}(\sigma^jx)\ge1-\epsilon$
  for all $k\ge1$ and $i=1,2$. Consequently we get
\begin{align*}
    \frac1{n_k(1)}\sum_{j=0}^{n_k(1)-1}\vfi(\sigma^jx)
    &\le
    (1-\epsilon)(\alpha_1+\epsilon)+\epsilon\|\vfi\|_0=: a_1 \qand
    \\
    \frac1{n_k(2)}\sum_{j=0}^{n_k(2)-1}\vfi(\sigma^jx)
    &\ge
    (1-\epsilon)(\alpha_2-\epsilon)-\epsilon\|\vfi\|_0=:a_2.
\end{align*}
Moreover, $a_1<a_2$ if, and only if,
$\frac{2\epsilon}{1-\epsilon}(\|\vfi\|_0+1-\epsilon)<\alpha_2-\alpha_1$
which is true for all small enough $\epsilon>0$. Hence
$\wt\vfi(x)$ does not exist.
\end{proof}

\subsubsection{Wild historic points are points with maximal
  oscillation}
\label{sec:wild-historic-points-1}

For any $x\in\XX$ we follow \cite{DGS76} and denote by
$V(x)$ the set of all weak$^*$ accumulation points of
$n^{-1}\sum_{j=0}^{n-1}\delta_{\sigma^jx}$. We write
$\PP_\sigma(\XX)$ for the family of all $\sigma$-invariant
Borel probability measures on $\XX$ endowed with the
weak$^*$ topology. Since $(\XX,\sigma)$
has the specification property, we are in the setting of
\cite[Propositions 21.8 \& 21.18]{DGS76}.

\begin{proposition}{\cite[Proposition 21.8]{DGS76}}
  \label{pr:DGS21.18}
  The set of measures concentrated on periodic orbits is
  dense in $\PP_\sigma(\XX)$.
\end{proposition}

\begin{proposition}{\cite[Proposition 21.18]{DGS76}}
  \label{pr:DGS21aa}
  The set of points with \emph{maximal oscillation}, that
  is, those $x$ for which $V(x)=\PP_\sigma(\XX)$, form a
  generic subset of $\XX$.
\end{proposition}

We now show that 

\begin{lemma}\label{le:wildmaxosc}
  Every point in the generic subset $Y\subset\W_\sigma$ is a
  point with maximal oscillation. Reciprocally, every point
  with maximal oscillation is a wild historic point.
\end{lemma}

\begin{proof}
  If $x\in\XX$ is a point with maximal oscillation, then
  given any periodic point $\un p\in\per(\sigma)$ we can
  find a sequence $n_k\nearrow\infty$ so that
  $\frac1{n_k}\sum_{j=0}^{n_k-1}\delta_{\sigma^jx}
  \xrightarrow[k\to\infty]{w^*}\frac1{\pi(\un
    p)}\sum_{j=0}^{\pi(\un p)-1}\delta_{\sigma^j\un p}$.
  In particular, we get that $\tau_xU=1$ for every
  neighborhood $U$ of
  $\{\un p,\sigma\un p,\dots,\sigma^{\pi(\un p)-1}\un p\}$
  and, for every small enough neighborhood $V$ of $\un p$,
  we get $ \tau_xV=\pi(\un p)^{-1}$. Moreover, $V$ contains
  distinct periodic points with all periods larger than $\pi(\un
  p)+\ell$ for some $\ell\in\ZZ^+$. Hence $\eta_x
  V\ge\sum_{k\ge\pi(\un p)+\ell}\frac1k=\infty$. Since
  periodic points are dense in $\XX$, this shows that
  $x\in\W_\sigma$.

  If $x\in Y$, then we have $\tau_xU=1$ for every
  neighborhood $U$ of
  $\cO_\sigma(\un p)=\{\un p,\sigma\un
  p,\dots,\sigma^{\pi(\un p)-1}\un p\}$
  of any given periodic point $\un p \in\per(\sigma)$, by
  Remark~\ref{rmk:unitmass}. Moreover, for every small
  enough neighborhood $V$ of $\un p$, we get
  $\tau_xV=\pi(\un p)^{-1}$. Hence, given a nested fundamental
  family $(U_k)_{kge1}$ of neighborhoods of $\cO_\sigma(\un p)$ and
  $(V_k)_{k\ge1}$ of $p$ we can find $n_k\nearrow\infty$ so that
  \begin{align*}
    \frac1{n_k}\sum_{j=0}^{n_k-1}\delta_{\sigma^jx}U_k\ge1-\frac1k,
    \qand
    \frac1{n_k}\sum_{j=0}^{n_k-1}\delta_{\sigma^jx}V_k\ge\frac1{\pi(\un
    p)}-\frac1k,
    \quad \forall k\ge1.
  \end{align*}
  Thus denoting $\mu$ a weak$^*$ accumulation point of
  $\frac1{n_k}\sum_{j=0}^{n_k-1}\delta_{\sigma^jx}$, we
  obtain $\mu(U_k)=1$ and $\mu(V_k)=\frac1{\pi(\un p)}$ for
  all $k\ge1$, and conclude that
  $\mu=\frac1{\pi(\un p)}\sum_{j=0}^{\pi(\un
    p)-1}\delta_{\sigma^j\un p}$.

  Because $\un p\in\per(\sigma)$ was arbitrary, we have show
  that $V(x)$ contains all probability measures supported on
  periodic points. Since these periodic measures are dense
  in $\PP_\sigma(\XX)$, we conclude that
  $V(x)=\PP_\sigma(\XX)$ and $x$ has maximal oscillation.
\end{proof}


\begin{example}\label{ex:extremenonnormal}
  The ``extreme non-normal numbers'' and ``extreme
  non-normal continued fractions'' studied by Olsen in
  \cite{Ols04,Ols03} are wild historic points, and so are
  generic subsets of the interval $[0,1]$. Since Liouville numbers
  (see e.g. Niven \cite{niven56} for a classical
  introduction) also form a generic subset of the real line,
  we have that generically all Liouville numbers are wild
  historic and extreme non-normal points for the maps
  $T_b:[0,1]\circlearrowleft, x\mapsto bx\bmod1$ for all
  positive integers $b\ge2$.

  In fact, extreme non-normal numbers are given by sequences
  $(a_n)_n\in\XX:=\NN^\NN$ so that the frequency of a sequence
  $b\in\NN^h$ of lenght $h\in\NN$ in the first $\ell$
  elements
  \begin{align*}
    \Pi\big((a_n)_n,b,\ell\big)=\frac1{\ell}\#\{1\le j\le \ell-h:
    a_{j}=b_1,a_{j+1}=b_2,\dots,a_{j+h-1}=b_h\}
  \end{align*}
  accumulates, when $\ell\to\infty$, on all the possible
  frequency vectors $p\in
  S_h\subset\Delta_h=\{(p_b)_{b\in\NN^h}:
  \sum_{b\in\NN^h}p_b=1\}$ of sequences of $h$ symbols, for
  each $h\ge1$; cf. the notion of points with maximal
  oscillation.

  An extremely non-normal admissible sequence $\underline
  x=(x_i)_{i\ge0}\in\XX$ visits every given fixed cylinder
  $[\underline p_k]_h$ infinitely many times. Hence, for
  every $h,m\ge1$ there exists $n\ge1$ such that
  $\sigma^{\ell_m}\underline x\in[p_h]_{\ell_m}$. We thus
  obtain (P1) and (P2) as in
  Subsection~\ref{sec:constr-wild-histor} and since this
  holds for infinitely many values of $m$, we also get
  (\ref{eq:cotacilindro}).

  The relation with number theory and the maps $T_b$ is
  given by the partition
  $\{J_i=[(i-1)/b,i/b[: i=1,\dots, b\}$ of the unit interval
  into equally sized intervals; and the connection with the
  Gauss map $G$ is provided by the partition
  $\{L_i=]1/(i+1),1/i], i\ge1\}$. We associate to a sequence
  $\un{a}=(a_n)_{n\ge1}\in\YY=\{1,\dots,b\}^\NN$ or
  $\un{b}=(b_n)_{n\ge1}\in\XX=\NN^\NN$ the point
  $h(\un{a})=\cap_{n\ge1}T_b^{-n}\ov{J_{a_i}}$ and
  $g(\un{b})=\cap_{n\ge1}G^{-n}\ov{L_{b_i}}$; obtaining
  surjective continuous maps $h:\YY\to[0,1], g:\XX\to[0,1]$
  such that $T_b\circ h=h\circ\sigma_{\YY}$ and
  $G\circ g=g\circ\sigma_{\XX}$. We can then use the
  extremely non-normal sequences to build the extremely
  non-normal numbers and continued fractions; see below for
  instances of similar constructions in other settings.
\end{example}

\begin{example}
  \label{ex:Naudot}
  The modified Bowen example from~\cite{JNY09} shows that the orbits
  of each point $x$, in the interior of the plane curve formed by the
  heteroclinic orbits connecting the fixed \emph{non-hyperbolic}
  saddle points $A,B$,  have time averages of ``type $B_2$''
  for each continuous observable $\vfi:\RR^2\to\RR$ with
  $\vfi(A)\neq\vfi(B)$. This is an extremely slow oscillating behavior
  of time averages that is preserved by all higher order averages;
  see~\cite{JNY09} for more details. These authors show that the Bowen
  Example~\ref{ex:Boweneye} with hyperbolic saddles does not admit
  such behavior.

  Since the modified Bowen example from~\cite{JNY09} has the same
  phase portrait as Example~\ref{ex:Boweneye}, we still have
  $\supp\eta_x=\{A,B\}$ as a consequence of Remark~\ref{rmk:tau0}
  together with Lemma~\ref{Lema005454a}. Moreover
  $\eta_x\big(B_{2\epsilon}(A)\setminus B_\epsilon(A))=0$ for all
  small enough $\epsilon>0$. Thus, by definition of $\eta_x$, we
  conclude that $\eta_x(\{A\})\le1$ (and likewise
  $\eta_x(\{B\})\le1$).

  Hence, for this modified Bowen example, $x$ is still an historic and
  not wild historic point for the time-$1$ map. This shows that
  ``type $B_2$'' orbits defined in~\cite{JNY09} do
  not always exhibit wild historic behavior.

  However, the construction in~\cite[Example 2]{JNY09} of an orbit of
  ``type $B_2$'' in the full shift with finitely many symbols is
  parallel to the one in Subsection~\ref{sec:constr-wild-histor}, and
  so all wild historic points in a full shift with finitely many
  symbols are of ``type $B_2$''. So we can loosely write
  $\cW_\sigma\subset B_2$ in this setting.
\end{example}

\subsection{Open continuous and transitive expansive maps}
\label{sec:open-contin-expans}

Let $f:X\to X$ be a continuous map of a compact metric
space. We say that $f$ is \emph{positively expansive}, that
is
\begin{align*}
  \exists \delta>0 : \big( d(f^nx,f^ny)\le\delta, \forall
  n\ge0\big) \implies x=y.
\end{align*}
An apparently stronger notion is that of
distance-expanding. We say that $f$ is
\emph{distance-expanding} if there are constants $\lambda>1,
\eta>0$ and $n\ge0$ so that for all $x,y\in\XX$
\begin{align*}
  d(x,y)\le2\eta\implies d(f^nx,f^ny)\ge\lambda d(x,y).
\end{align*}
We are then in the setting of
\begin{theorem}{\cite[Theorem 4.6.1]{PrzyUrb10}}
  \label{thm:expansiveexpanding}
If a continuous map $f:X\to X$ of a compact metric space
is positively expansive, then there exists a metric on
$X$, compatible with the topology, such that $f$ is
distance-expanding with respect to this metric.
\end{theorem}

Finally, open distance-expanding maps on compact metric
spaces admit Markov partitions \cite[Theorem
4.5.2]{PrzyUrb10} and hence the dynamics of
these maps is semiconjugated to a Topological Markov Chain,
as follows.

\begin{theorem}{\cite[Theorems 4.3.12 \& 4.5.7]{PrzyUrb10}}
  Let $f:X\to X$ be an open distance-expanding map. Then
  there exists a $d\times d$ matrix
  $A\in\{0,1\}^{d\times d}$ such that the corresponding
  one-sided topological Markov Chain $\XX=\Sigma_A$ with the
  left shift map $\sigma:\XX\circlearrowleft$ admits a
  continuous surjective mapping $\pi:\XX\to X$ such
  that $\pi\circ\sigma=f\circ\pi$ and a generic subset
  $Z\subset X$ so that
  $\pi\mid_{\pi^{-1}(Z)}:\pi^{-1}(Z)\to Z$ is
  injective. 

  Moreover, $f$ admits a countable dense subset of periodic
  points and if, in addition, $f$ is transitive, then $f$ is
  topologically mixing.
\end{theorem}

Thus $\pi$ is a semiconjugation, generically a conjugation,
and we have $\tau^f_{\pi x}=\pi_*(\tau^\sigma_x)$ for
$x\in\XX$, where $\tau^f_y$ and $\tau^\sigma_x$ represent
the measures $\tau$ with respect to the $f$ and $\sigma$
dynamics, respectively.  Therefore,
$\pi(\W_\sigma\cap \pi^{-1}Z)\subset\W_f\cap Z$ is a generic
subset of $X$ if $\W_\sigma$ is a generic subset of $\XX$.

To conclude the proof of item (2) of
Theorem~\ref{mthm:genericallywild} we note that, \emph{if
  $f$ is an open continuous expansive and topologically
  transitive map}, then we can find a compatible metric with
respect to which \emph{$f$ becomes topologically mixing and
  conjugated on a generic subset to the mixing topological
  Markov Chain $\sigma:\XX\circlearrowleft$.} In this
setting, we know that $\W_\sigma$ is a generic subset of
$\XX$ and so $f$ admists a generic subset $\W_f$ of wild
historic points.

\subsection{Transitive local homeomorphisms with induced
  full branch Markov map}
\label{sec:transit-local-homeom}

For item (3) of Theorem~\ref{mthm:genericallywild} we recall
that $f:X_0\to X$ is a local homeomorphism where $X_0$ is an
open dense subset of a compact metric space $X$. We assume that
$f$ is topologically transitive and that there exists a open connected
subset $\Delta\subset X_0$ and an induced full branch Markov
map $F:G\to\Delta\supset G$. This means
\begin{enumerate}
\item[(a)] there exists a function $R:G\to\ZZ^+$ where $G$ is an
  open dense subset of $\Delta$ and $Fz=f^{R(z)}(z)$ for all
  $z\in G$;
\item[(b)] there exists an at most denumerable partition
  $\cP=\{\Delta^i\}_{i\ge1}$ of $G$ such that $R\mid
  \Delta^i\equiv R_i$ is constant on every element of $\cP$;
\item[(c)] $F_i=F\mid_{\Delta^i}:\Delta^i\to\Delta$ is an expanding
  homeomorphism: there exists $\sigma>1$ such that
  $d(F_ix,F_iy)\ge\sigma d(x,y)$ for all $x,y\in\Delta^i, i\ge1$.
\end{enumerate}
The assumptions on $F$ ensure that there exist a countable
dense subset of periodic orbits for $F$ in $\Delta$ and,
moreover, there exists a surjective continuous map
$h:\XX\to\Delta$ such that $h\circ \sigma=F\circ h$, where
$\sigma:\XX\circlearrowleft$ is the full shift with
countable number of symbols. Indeed, we just have to define
$h(\theta)=\cap_{n\ge0}\ov{F^{-n}\Delta^{\theta_n}}$ for
$\theta\in\XX$, which is well-defined by item (c)
above. Setting
$Z=\Delta\setminus\big(\cup_{n\ge1}F^{-1}(\Delta\setminus
G)\big)$
we have a generic subset of $\Delta$ such that
$h\mid_{h^{-1}Z}: h^{-1}Z\to Z$ is injective. Thus $h$ is a
conjugation between generic subsets.

We write
$\tau_x^F(A)=\limsup_{n\to\infty}\frac1n\sum_{j=0}^{n-1}
\delta_{F^jx}(A)$
for all $A\subset\Delta$ and also $\tau_y^f, \tau^\sigma_z$
for the same notions for $f$-orbits of $y\in X$ and
$\sigma$-orbits of $z\in\XX$, respectively. We also write
$\W_F$ for the set of wild historic points for
$F:G\to\Delta\supset G$, $\W_f$ the set of wild historic
points for $f:X_0\to X$ and $\W_\sigma$ for the set of wild historic
points of $\sigma:\XX\circlearrowleft$.

It is easy to see that $\tau^F_{h x}=h_*\tau^\sigma_x$ for
all $x\in\XX$ and consequently that
$h(\W_\sigma\cap h^{-1}Z)\subset\W_F\cap Z$ is a generic
subset of $\Delta$. Since $\cup_{n\ge0}f^{-n}\Delta$ is
dense in $X$ by transitivity, it is enough to show that
$\W_F\subset\W_f$ to conclude that $\W_f$ is a generic
subset of $X$, because $\W_f$ is $f$-invariant.

\begin{lemma}
  \label{le:W_F}
  $\W_F\subset \W_f$.
\end{lemma}

\begin{proof}
  We observe that every $p\in\per(F)$ with period $\pi_F(p)$
  is such that $p\in\per(f)$ with period
  $\pi_f(p)=S^f_{\pi_F(p)}R(p)=\sum_{i=0}^{\pi_F(p)-1}R(F^ip)$.
  In addition, since $\ov{\per(F)}=\Delta$ and by
  transitivity $\ov{\cup_{n\ge0}f^n\Delta}=X$, we conclude that
  $\ov{\per(f)}=X$.

  From Remark \ref{rmk:unitmass}, there exists a
  topologically generic subset $Y$ of $\W_F$ and
  $\gamma:\per(F)\to\RR^+$ such that for all $y\in Y$ we get
  $\tau^F_y(U)\ge \gamma(p)$ for each neighborhood $U$ of
  $p\in\per(F)$. By the relation between the periods of
  $p\in\per(F)$ with respect to $F$ and to $f$, we obtain
  $\tau^f_y(U)\ge\gamma(p)\frac{\pi_F(p)}{\pi_f(p)}
  =\gamma(p)\frac{\pi_F(p)}{S^f_{\pi_F(p)}R(p)}=\xi(p)>0$.

  Since $p\in\per(F)$ is also a periodic point for $f$ and
  clearly $\xi(p)=\xi(fp)$, we have obtained $\xi:P\to\RR^+$
  so that, for every neighborhood $V$ of a point $q$ in the
  dense subset $P:=\cup_{i\ge0}f^i\per(F)$ of $\per(f)$, it
  holds $\tau^f_y(V)\ge\xi(q)$. This is enough to conclude,
  again by Remark \ref{rmk:unitmass}, that $y\in Y$ belongs
  to $\W_f$.
\end{proof}

To prove item (3) of Theorem~\ref{mthm:genericallywild},
just recall that $\W_f$ is $f$-invariant. Since
$\W_f\supset\W_F$ is residual in $\Delta$ (i.e., it contains
a denumerable intersection of open and dense subsets of
$\Delta$) and $f$ is a transitive local homeomorphism on an
open dense subset, then
$\W_f\supset\cup_{n\ge1}f^n(\W_f\cap\Delta)$ is a residual
subset of $X$.

\subsection{Special semiflows over local homeomorphisms}
\label{sec:special-flows-over}

For item (4) of Theorem~\ref{mthm:genericallywild}, we keep $f$ as in
the previous setting of Subsection~\ref{sec:transit-local-homeom},
take a measurable function $r:X\to[r_0,+\infty]$, where we fix $r_0>0$
and assume $r\mid_{X_0}<\infty$, and consider the special semiflow
over $f$ with roof function $r$, that we denote by
$\phi_t:X_f^r\circlearrowleft$; see e.g. \cite{KH95} or \cite{PM82}
for the definition and basic properties of special/suspension
flows. Here $X^r=\{(x,s)\in X\times[0,+\infty]: 0\le s <r(x)\}$ is the
ambient space of the flow and $X\simeq X\times\{0\}$ becomes a
cross-section for $\phi_t$.

In this setting we analogously define a pre-measure
$\tau_{(x,s)}^\phi(E) =
\limsup\limits_{T\to\infty}\frac1T\int_0^T1_{E}\big(\phi_t(x,s)\big)\
dt$ and then build the measure $\eta^\phi_{(x,s)}$ from the
pre-measure as explained before.  We wrote $\W_\phi$ for the
set of wild historic points for $\phi_t$.

If we assume that $r\le r_1$ for some constant $r_1>r_0$, then
for a given $x\in \wt{X_0}$ the hitting times at $X$ of the
forward $\phi$-orbit of $(x,0)$ are
$T_n=T_n(x)=S_n^fr(x)=\sum_{i=0}^{n-1}r(f^ix)\in [r_0n,r_1n]$ for all
$n\ge1$. For an elementary open set $A\times I$ of $X^r$,
where $A$ is a open in $X$ and $I$ is an open interval in
$\RR$ so that $A\times I\subset X^r$, we have
\begin{align}
  \tau_{(x,0)}^\phi(A\times I) &=\nonumber
  \limsup_{T\to\infty}\frac1T\int_0^T1_{A\times
    I}\big(\phi_t(x,0)\big)\ dt
  \\
  &\ge
  \limsup_{n\to\infty}\frac{n}{T_n}\cdot\frac1n\sum_{i=0}^{n-1}1_A\big(f^jx\big)
  \cdot\leb(I) = \frac{\leb(I)}{r_1}\tau_x^f(A)   \label{eq:wildphi}.
\end{align}
This shows that $\W_f\subset \W_\phi$ and so $\W_\phi$ contains
a residual subset of $X^r$ from item (3) of
Theorem~\ref{mthm:genericallywild} already proved.


\subsection{Stable sets of Axiom A basic sets}
\label{sec:basic-sets-axiom}

For item (5) of Theorem~\ref{mthm:genericallywild} we use
known results from the theory of Hyperbolic Dynamical
Systems; for the setting and relevant definitions, see
\cite{Sm67,Bo73,Bo75,BR75}.  

\subsubsection{The Axiom A diffeomorphism case}
\label{sec:axiom-diffeom-case}

We recall that each basic set
$\Lambda$ of an Axiom A diffeomorphism $f:M\to M$ of a
compact manifold is a finite pairwise disjoint union
$\cup_{i=1}^k\Lambda_i=\cup_{i=1}^{n-1}f^i(\Lambda_1)$ of
compacta such that $f^n(\Lambda_1)=\Lambda_1$, and $f^n\mid
\Lambda_1$ is semiconjugated to a subshift of finite type
$\XX$. That is, there exists a surjective continuous map
$h:\XX\to\Lambda_1$ such that $h\circ\sigma=f\circ h$.

This ensures that there is a generic subset
$Y_1=h(\W_\sigma)$ of wild historic points in $\Lambda_1$
for the action of $f^n$, which in turn generates a generic
subset $\hat Y=Y_1\cup f(Y_1)\cup\dots\cup f^{n-1}(Y_1)$ of
wild historic points of $\Lambda$ for the action of $f$.  In
addition, since points in the \emph{stable set of $\Lambda$}, given by
\begin{align*}
  W^s(\Lambda)=\{z\in M:d(f^n(z),\Lambda)\xrightarrow[n\to+\infty]{}0\},
\end{align*}
belong to the stable set of some point of $\Lambda$,
that is, $W^s(\Lambda)=\cup_{x\in\Lambda}W^s(x)$, where
\begin{align*}
  W^s(x)=\{z\in M:d(f^n(z),f^n(x))\xrightarrow[n\to+\infty]{}0\};
\end{align*}
and points in the stable set of $x$ have the same asymptotic sojourn
times as the orbit of $x$; we have that
$W^s(\hat Y)=\cup_{x\in\hat Y}W^s(x)$ is a topologically generic
subset of $W^s(\Lambda)$ formed by wild historic points, and
$\W_f\supset Y$.




\subsubsection{The Axiom A vector field case}
\label{sec:vector-field-case-1}

For a basic set $\Lambda$ of an Axiom A vector field $X$ on
a compact manifold, we analogously have that $\Lambda$ is
semiconjugated to a suspension flow
$\phi_t:\XX^r\circlearrowleft$ over a two-sided subshift of
finite type $\XX$ with a bounded roof function
$r:\XX\to[r_0,r_1]$ for some $0<r_0<r_1, r_0,r_1\in\RR$; see
\cite{Bo73,BR75}.

From item (3) of Theorem~\ref{mthm:genericallywild} already
proved, we see that $\W_\phi$ is a
topologically generic subset of $\XX^r$. If
$h:\Lambda\to\XX^r$ is the semiconjugation between the
actions of $X$ on
$\Lambda$ and $\phi_t$ on $\XX^r$, we get
$\W_X\cap\Lambda\supset h^{-1}\W_\phi$. Hence the set of
wild historic points on $\Lambda$ is again a topologically
generic subset.

This completes the proof of item (5) of
Theorem~\ref{mthm:genericallywild}.

\subsection{Expanding measures}
\label{sec:expanding-measures}
For item (6) of Theorem~\ref{mthm:genericallywild}, we use
\cite{Pinho2011}. 
We recall that a $C^{1+\alpha}$-map $f:M\setminus\cC\to M$
of a compact manifold $M$ is \emph{non-flat} if $f$ is a
local diffeomorphism everywhere except at a
\emph{non-degenerate singular/critical set} $\cC$, that is,
$M\setminus\cC$ is open and dense in $M$ and there are
$\beta,B>0$ so that
 \begin{itemize}
 \item[(S1)]
\hspace{.1cm}$\displaystyle{\frac{1}{B} d(x,\cS)^{\beta}\leq
\frac{\|Df(x)v\|}{\|v\|}\leq B d(x,\cC)^{-\beta}}$;
 \item[(S2)]
\hspace{.1cm}$\displaystyle{\left|\log\|Df(x)^{-1}\|-
\log\|Df(y)^{-1}\|\:\right|\leq
B\frac{ d(x,y)}{d(x,\cC)^{\beta}}}$;
 \end{itemize}
 for every $x,y\in M\setminus \cC$ with $d(x,y)<d(x,\cC)/2$
 and $v\in T_x M\setminus\{0\}$.

An invariant \emph{expanding measure} for $f$ is a
probability measure $\mu$ satisfying
\begin{itemize}
\item \emph{non-flatness:} $\mu(\cC)=0$, $f_*\mu\ll\mu$,
  $\mu$ admits a Jacobian $J_\mu f(x)$ with respect to $f$ well defined
  and positive $\mu$-a.e. and, for $\mu$-a.e. $x,y\in
  M\setminus \cC$ with $d(x,y)<d(x,\cC)/2$, we have
  \begin{align*}
    \left| \log\frac{J_\mu f(x)}{J_\mu f(y)}\right|
    \le\frac{B}{d(x,y)^\beta}\cdot d(x,y);
  \end{align*}
\item \emph{non-uniformly expansion}:
there exists $c>0$ such that for $\mu$-a.e. $x$
\begin{align*}
  \limsup_{n\to+\infty}\frac1n \sum_{j=0}^{n-1}\log\big\| Df(f^jx)^{-1} \big\| \le -c;
\end{align*}
\item \emph{slow recurrence $\cC$}: for every $\epsilon>0$ there
  exists $\delta>0$ such that for $\mu$-a.e. $x$
\begin{align*}
  \limsup_{n\to\infty} \frac1n \sum_{j=0}^{n-1}\big|
  \log d_{\delta}(f^jx,\cS) \big| < \epsilon;
\end{align*}
where $d_{\delta}(x,\cC)$ denotes the $\delta$-\emph{truncated
  distance} from $x$ to $\cC$ defined as $d_{\delta}(x,\cC)=d(x,\cC)$
if $d(x,\cC) \leq \delta$ and $d_{\delta}(x,\cC) =1$ otherwise.
\end{itemize}

Then \cite[Theorem B]{Pinho2011} ensures, in particular,
that every expanding invariant measure $\mu$ for a non-flat
map $f$ of a compact manifold admits an induced full branch
Markov map defined on an open subset of $\supp\mu$. Hence
the set $\W_f\cap\supp\mu$ is a topologically generic subset
of $\supp\mu$ by item (2) of
Theorem~\ref{mthm:genericallywild} already proved. This
completes the proof of item (6) of
Theorem~\ref{mthm:genericallywild}.

\begin{example}\label{sec:expand-measur}
  As detailed in \cite{Pinho2011} the class of expanding
  measures presented above encompasses
  \begin{enumerate}
  \item the absolutely continuous invariant probability
    measure for non-uniformly expanding maps introduced in
    \cite{ABV00,Al00} including the Viana maps from
    \cite{Vi97};
  \item all piecewise
    expanding $C^{1+\alpha}$ maps of the interval, which
    include the Lorenz-like transformations of the interval
    studied in \cite{APPV};
  \item the absolutely continuous invariant probability measures for
    smooth multidimensional expanding maps studied
    in~\cite{GorBoy89,Sa00}.
  \end{enumerate}
\end{example}

\subsection{Hyperbolic measures for diffeomorphisms and flows}
\label{sec:hyperb-measur-diffeo}

For item (7) of Theorem~\ref{mthm:genericallywild}, we use
the following well-known result from the theory of
non-uniform hyperbolicity; see~\cite{katok80},
\cite[Supplement]{KH95} and~\cite{BarPes2007} for a modern
presentation of this theory. We write $\per_h(f)$ for the
set of hyperbolic periodic points of the map $f$.

\begin{theorem}{\cite[Theorem S.5.3,
    pp. 694-695]{KH95}.}\label{thm:homoclinicmeasure}
  If $\mu$ is an ergodic hyperbolic continuous measure (that is, $\mu$
  has no atomic part) for a $C^{1+\alpha}$ diffeomorphism
  $f:M\circlearrowright$ of a compact manifold, for some given fixed
  $\alpha>0$, then there exists $p\in\per_h(f)$ such that
  $\supp(\mu)\subset H(p)=\overline{W^s(p)\pitchfork W^u(p)}$, that
  is, the support of $\mu$ is contained in the homoclinic class of
  $p$. In particular, $\supp(\mu)\subset\overline{\per_h(f)}$.
  \end{theorem}

\subsubsection{The diffeomorphism case}
\label{sec:diffeomorphism-case}

To prove item (6) of Theorem~\ref{mthm:genericallywild}, let
$\mu$ be an ergodic non-atomic hyperbolic probability
measure for a $C^{1+\alpha}$ diffeomorphism
$f:M\circlearrowright$ of a compact manifold, for some
$\alpha>0$. 

The Birkhoff-Smale Theorem~\cite{Sm65} (see also~\cite[Theorem
5.5]{Sm67}) ensures that every neighborhood of a homoclinic point
intersects a horseshoe.  A compact $f$ invariant subset $\Lambda$ is a
\emph{horseshoe} if there are $s,k\in\ZZ^+$ such that $\Lambda$
decomposes as a disjoint union $\Lambda_0\cup\dots\cup\Lambda_{k-1}$
satisfying $f^k(\Lambda_i)=\Lambda_i$,
$f(\Lambda_i)=\Lambda_{i+1\bmod k}$ and $f^k\mid \Lambda_0$ is
topologically conjugated to a full shift in $s$ symbols.  This
implies, after Theorem~\ref{thm:homoclinicmeasure}, that densely in
the support of $\mu$ we can find horseshoes.

  From item (1) already proved, we have that densely in
  $\supp(\mu)$ there are points with wild historic behavior in
  some horseshoe $\Lambda_z\subset\supp(\mu)$: let
  $z_n\in\supp(\mu)$ be an enumeration of such dense
  subset. This means in particular that $\eta_{z_n}(A)=+\infty$
  for all open subsets $A$ of $\Lambda_{z_n}$, for all $n\ge1$.

  For $z\in\supp(\mu), n\ge1$ and $\epsilon>0$ let
  \begin{align*}
    B(z,n,\epsilon)
    =
    \{x\in\supp(\mu):d(f^ix,f^iz)<\epsilon,\forall 0\le i <n\}
  \end{align*}
  be the $(n,\epsilon)$-ball (dynamical ball) around $z$.
  We consider
  $Y_i=\cap_{n\ge1}\cup_{k\ge n}f^{-k}B(z_i,n,1/n)$ for each
  $i\in\ZZ^+$. By construction, for a given $w\in Y_i$ with
  $i\in\ZZ^+$, we can find $k_n$ such that
  $f^{k_n}w\in B(z_i,n,1/n)$ and $k_n>e^n$. So, for any
  given $\epsilon>0$ and $y\in\Lambda_{z_i}$ and all big
  enough $n$ so that $1/n<\epsilon$, we get
  \begin{align*}
    \frac1{n+\ell}\sum_{j=0}^{n+\ell-1}1_{B(y,2\epsilon)}(f^jw)
    \ge
    \frac{\ell}{n+\ell}\cdot\frac1{\ell}\sum_{j=0}^{\ell-1}1_{B(y,\epsilon)}(f^jz_i)
  \end{align*}
  for all $0\le\ell<k_n$. That is
  $\tau_w(B(y,2\epsilon))\ge\tau_{z_i}(B(y,\epsilon))$.  In
  particular, this means that
  $\eta_w\mid\Lambda_{z_i}\ge\eta_z\mid\Lambda_{z_i}$ and so
  $\eta_w(B(z_i,\epsilon))=+\infty$.

  Now let us consider $Y=\cap_{n\ge1}\cap_{i=1}^n\cup_{k\ge
    n}f^{-k}B(z_i,n,1/n)$.  Since $\mu$ is $f$-ergodic, we
  have that $Y$ is a denumerable intersection of the open
  and dense subsets $\cap_{i=1}^n\cup_{k\ge
    n}f^{-k}B(z_i,n,1/n)$ of the compact set $\supp(\mu)$,
  thus $Y$ is residual in $\supp(\mu)$. Moreover, for $w\in
  Y$ and $A$ a non-empty open subset of $\supp\mu$, there
  exist $i\ge1, \epsilon>0$ such that
  $B(z_i,\epsilon)\subset A$, thus $\eta_w(A)=+\infty$,
  showing that every $w\in Y$ has wild historic behavior:
  $W_f\cap\supp(\mu)\supset Y$. This completes the proof for
  the support of a non-atomic ergodic hyperbolic probability
  measure for a  $C^{1+}$ diffeomorphism.

\subsubsection{The vector field case}
\label{sec:vector-field-case}

Let $\mu$ be a non-atomic ergodic hyperbolic probability
measure for a $C^{1+}$ vector field $X$ of a compact
manifold $M$. Then a flow version of
Theorem~\ref{thm:homoclinicmeasure} also holds, that is,
$\supp\mu$ is contained in the the closure of the hyperbolic
periodic points that have transverse homoclinic points and,
by ergodicity, it is in fact contained in an homoclinic
class of some hyperbolic periodic point $p$ for the flow
$\phi_t$ generated by $X$.

Hence, the version of the Birkhoff-Smale Theorem for vector fields
(see e.g.~\cite[Part II]{Sm67} and~\cite{shil67a}) guarantees that
every neighborhood of a homoclinic point intersects the
\emph{suspension} of a horseshoe, by a bounded roof function. We
recall that each special/suspension flow with bounded roof function
over a horseshoe admits a topologically generic subset of wild
historic points, from item (4) of Theorem~\ref{mthm:genericallywild}
already proved.

Then the same argument as in Subsection~\ref{sec:diffeomorphism-case}
shows that there exists a topologically generic subset of wild
historic points in $\supp\mu$. Indeed, according the results on
existence, uniqueness and continuity of solutions of Ordinary
Differential Equations, for each $x\in M$ and every $T>0$ there exists
$\epsilon=\epsilon(x,T)>0$ so that the dynamical ball
  \begin{align*}
    B(x,T,\epsilon)=\{y\in M: d(\phi_tx,\phi_ty)<\epsilon, \forall
    0\le t\le T\}
  \end{align*}
  is an open neighborhood of $x$ and note that $\epsilon(x,T)\to0$ as
  $T\to+\infty$. Hence, from the Birkhoff-Smale Theorem and the
  existence of a generic subset of wild historic points for suspended
  horseshoes, we have $D\subset\supp\mu$ a denumerable dense subset of
  wild historic points, each point in some suspended horseshoe. Given
  an enumeration $\{z_i\}_{i\in\ZZ^+}$ of $D$, we consider the set
  $Y=\cap_{n\ge1}\cap_{i=0}^n\cup_{k\ge n} \phi_{-k}B(z_i,
  n,\epsilon(z_i,n))$ which is again topologically generic in
  $\supp\mu$. Given $y\in\supp\mu, x\in Y, z_i\in D$ and $\delta>0$,
  for all large enough $n\in\ZZ^+$ there exists $k=k_n>e^n$ so that
  $\phi_kx\in B\big(z_i,n,\epsilon(z_i,n)\big)$ and
  $\delta>\epsilon(z_i,n)$ and also
  \begin{align*}
    \frac1{n+\ell}\int_0^n1_{B(y,2\delta)}\big(\phi_tx\big)\,dt
    \ge
    \frac{\ell}{n+\ell}\cdot\frac1{\ell}
    \int_0^\ell1_{B(y,\delta)}\big(\phi_tz_i\big)\,dt
  \end{align*}
  for each $0\le\ell<k_n$. This shows that
  $\tau_x^\phi\big(B(y,2\delta)\big)\ge
  \tau_{z_i}^\phi\big(B(y,\delta)\big)$ and so, since
  $y\in\supp\mu$ and $\delta>0$ where arbitrarily chosen, we
  get $\eta_x^\phi\ge\eta_{z_i}^\phi$ for all
  $i\in\ZZ^+$. Thus $\eta_x^\phi(A)=+\infty$ for all each
  subset $A$ of $\supp\mu$, since $D=\{z_i\}_{i\in\ZZ^+}$ is
  dense in $\supp\mu$. We proved $\W_\phi\supset Y$.

  This finishes the proof of item (7) of
  Theorem~\ref{mthm:genericallywild} and completes the proof
  of Theorem~\ref{mthm:genericallywild}.

  \begin{example}
    \label{ex:hyperb-measur-Lorenz}
    Every geometric Lorenz attractor, the classical Lorenz
    attractor, and every singular-hyperbolic attractor (also
    known as Lorenz-like attractors, which generalize the
    notion of uniform hyperbolicity to invariant sets of
    flows containing hyperbolic singularities accumulated by
    regular orbits), admit an ergodic hyperbolic measure
    which is also physical and, in particular, non-atomic;
    see \cite{APPV,AraPac2010}. Hence, the set of wild
    historic points inside these attractors is a
    topologically generic subset. In addition, the stable
    set for this class of attractors also admits a
    topologically generic set of wild historic points.
  \end{example}

  \begin{example}
    \label{ex:hyp-meas-Rovella}
    The \emph{contracting Lorenz attractors} (or
    \emph{Rovella attractors}) admit a non-atomic ergodic
    hyperbolic physical measure: see~\cite{Ro93}
    and~\cite{mtz001}. Hence this class of persistent
    attractors contains a topologically generic subset of
    wild historic points as well as their stable set.
  \end{example}

  \begin{example}
    \label{ex:hyp-meas-sec-hyp}
    In higher dimensions, an extension of uniform
    hyperbolicity encompassing singular flows analogous to
    that of singular-hyperbolicity for $3$-flows is the
    notion of \emph{sectional-hyperbolic
      sets}~\cite{MeMor06}. Recently, it has been
    shown~\cite{LeplYa17,MetzMor15} that
    sectional-hyperbolic attractors, for flows in any
    dimension greater or equal to $3$, also admit a a
    non-atomic ergodic hyperbolic physical measure. Hence,
    this class of attractors satisfies the same properties
    of abundance of wild historic points as the class of
    singular-hyperbolic attractors in
    Example~\ref{ex:hyperb-measur-Lorenz}.
  \end{example}

    \begin{remark}
    \label{rmk:nospecification}
    Sumi, Varandas and Yamamoto \cite{SuVaYa15} have shown
    that sectional-hyperbolic attractors (which include
    singular-hyperbolic attractors as a special
    $3$-dimensional case) do not satisfy
    specification. Hence, the genericity of wild historic
    points is more general than the genericity of points
    with maximal oscillation in systems with specification.
  \end{remark}


\section{On strong historic behavior}
\label{sec:strong-histor-behavi}

Here we prove Theorem~\ref{mthm:openhistoric-heteroclinic}
obtaining a partial characterization of the behavior of
Example~\ref{ex:Boweneye}. We start by observing that if
$\eta_x$ is an atomic measure, then the atoms must be
preperiodic points of the transformation.

\begin{lemma}
  \label{le:P}
  If $\eta_x$ is purely atomic for some $x\in\XX$, with
  $f:\XX\to\XX$ a continuous map with finitely many
  pre-images, then every atom $P$ is a preperiodic point for
  $f$. In particular, if $f$ is a bijection, then $P$ is a
  periodic point for $f$.
\end{lemma}

\begin{proof}
  From Theorem~\ref{TheoremContinuous} we know that $\eta_x$
  is a $f$-invariant measure. Then
  $\eta_x(f(P))=\eta_x(f^{-1}(f(P)))\ge\eta_x(P)>0$ and so
  $f(P)$ is an atom of $\eta_x$. Since we assume that
  $\eta_x$ has finitely many atoms, then $f$ maps this
  finite set into itself, and each atom is preperiodic for $f$.
\end{proof}

Let us assume that the system $f:M\to M$ is a diffeomorphism
on a compact boundaryless manifold $M$ and that there are at
least a pair of hyperbolic periodic points $P,Q$ of $f$
satisfying, for some $\epsilon>0$
\begin{itemize}
\item[(H)] $\eta_x$ is atomic with two atoms $P,Q$ for every
  $x$ either in $B_\epsilon(P)\setminus(W^s_\epsilon(P)\cup
  W^u_\epsilon(P))$ or $ B_\epsilon(Q)\setminus(W^s_\epsilon(Q)\cup
  W^u_\epsilon(Q))$.
\end{itemize}
We will show that under these conditions we have the same
configuration as in Example~\ref{ex:Boweneye}.

\begin{theorem}
  \label{thm:2Bowen}
  Under assumption $(H)$ we have that $W^s(P)=W^u(Q)$ and
  $W^u(P)=W^s(Q)$, that is, $P$ and $Q$ have heteroclinic
  connections.
\end{theorem}

\begin{proof}
  Let us fix $\epsilon>0$ with the properties given in
  assumption $(H)$ and $x$ such that
  $\eta_x=a\delta_P+b\delta_Q$ with $a,b>0$.

  The assumption that $\eta_x$ is formed by precisely two
  atoms ensures that $P,Q$ are fixed points, for otherwise
  all the points in the orbit of $P,Q$ would also be atoms
  of $\eta_x$.

  The definition of $\eta_x$ ensures that
  $\tau_x(B_\delta(P))>0$ for all $\delta>0$. Hence we can
  find a sequence $k_i$ of iterates such that
  $x_{k_i}:=f^{k_i}(x)\to P$ when $i\to+\infty$. Since $P$
  is a hyperbolic fixed point, there exists
  $n_k\nearrow\infty$ and a point $y\in W^s_{\epsilon/2}(P)$
  such that $x_{n_k}\to y$ when $k\to\infty$.

  Writing $\sigma>1$ for the least expanding eigenvalue of
  $Df(P)$ and $\lambda=\sup\|Df\|$ we obtain for some
  $m_k\nearrow\infty$ 
  \begin{align}\label{eq:dist}
    d(x_{n_k},y)\le \sigma^{-m_k} \qand
    d(x_{n_k-i},y_{-i}) \le \sigma^{-m_k}\lambda^i, 
    \quad 0\le i\le n_k.
  \end{align}
  We claim that there exists $\ell\ge1$ such that
  $y_{-\ell}\in B_\epsilon(Q)$.

  Indeed, let us assume that $y_{-\ell}\not\in
  B_\epsilon(Q)$ for each $\ell\ge1$. Hence, for every big
  enough $k$, each visit $x_{n_k}$ to $B_\epsilon(P)$
  corresponds to a visit to
  $K=B_\epsilon(W^s_{\epsilon}(P))\setminus(
  B_\epsilon(P)\cup B_\epsilon(Q))$ of $x_{n_k-\ell}$ for
  some $\ell\ge1$; see Figure~\ref{fig:nearP}. Moreover from
  (\ref{eq:dist}) we can ensure that the set of $\ell$ for
  which $x_{n_k-\ell}$ visits $K$ has size proportional to
  $m_k$ for big $k$, that is
    \begin{align*}
     \frac1{m_k} \max\{\ell\ge1:
     \sigma^{-m_k}\lambda^{\ell}<\epsilon\}
     \xrightarrow[k\to\infty]{}-\frac{\log\sigma}{\log\lambda}=\xi.
    \end{align*}
Thus we obtain that $\tau_x(M\setminus(B_\epsilon(P)\cup
B_\epsilon(Q))) \ge \tau_x(K) \ge\xi\tau_x(B_\epsilon(P))>0$
which contradicts the assumption that the atoms of $\eta_x$
are $P,Q$ only. This proves the claim.

\begin{figure}[htpb]
  \centering
  \includegraphics[width=7cm]{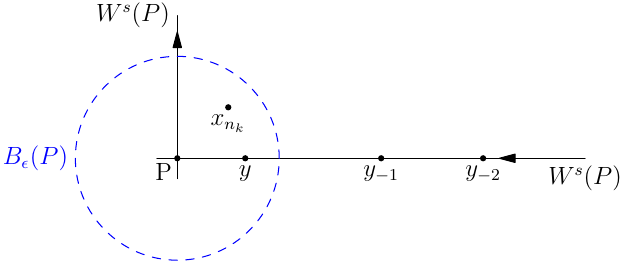}
  \caption{The position of $y$ and $y_\ell$ around $P$.}
  \label{fig:nearP}
\end{figure}

We note that $y_{-\ell}\in B_\epsilon(Q)$ does not satisfy
$(H)$ unless $y_{-\ell}\in W^u(Q)\cup W^s(Q)$. But because
$y\in W^s(P)$ we deduce that $y_{-\ell}\in W^u(Q)$.

\begin{figure}[htpb]
  \centering
  \includegraphics[width=7cm]{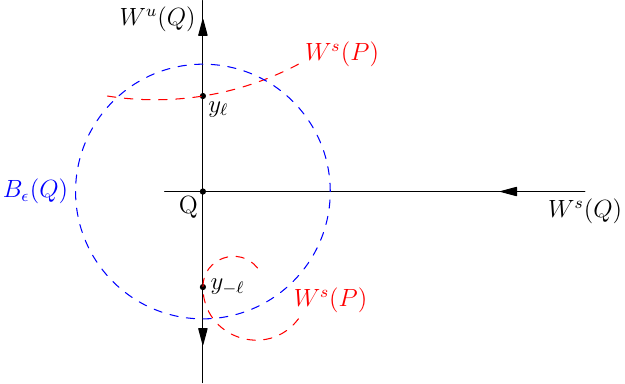}
  \caption{The position of $y_{-\ell}$ and $W^s(P)$ near
    $Q$; in the upper side the transversal situation, in the
    lower side the tangent situation.}
  \label{fig:nearQ}
\end{figure}

Moreover, see Figure~\ref{fig:nearQ}, we cannot have a
transversal intersection between $W^u(Q)$ and $W^s(P)$ at
$y_{-\ell}\in B_\epsilon(Q)$, for otherwise we would obtain
a point $z\in W^s(P)\cap B_\epsilon(Q)$ which contradicts
$(H)$. By the same reason, a tangency between $W^u(Q)$ and
$W^s(P)$ at $y_{-\ell}\in B_\epsilon(Q)$ is not
allowed. Hence the connected component of $W^s(P)$ in
$B_\epsilon(Q)$ containing $y_{-\ell}$ must be contained in
$W^u(Q)$. Thus, $W^u(Q)\subset W^s(P)$ since both invariant
manifolds are immersed submanifolds of $M$ and they are
uniquely defined in a neighborhood of $P,Q$.

An analogous reasoning provides $W^u(P)\subset
W^s(Q)$. Since the dimensions of the stable and unstable
manifolds of each $P,Q$ are complementary, we conclude that
the dimensions of $W^s(P)$ and $W^u(Q)$ are the same, and so
$W^s(P)=W^u(Q)$. Similarly we arrive at $W^u(P)=W^s(Q)$.
\end{proof}

\begin{remark}
  \label{rmk:finiteatoms}
  The argument in the proof of Theorem~\ref{thm:2Bowen} can
  be easily adapted to a setting where the number of atoms
  of $\eta_x$ is a finite set of hyperbolic periodic points
  for all $x$ in a neighborhood of each periodic point such
  that $x$ leaves that neighborhood in the future and
  past. That is, for all $x$ in a neighborhood of the
  periodic points with the exception of its invariant
  manifolds.
  We then obtain heteroclinic connections between a finite
  family of hyperbolic periodic points as
  in~\cite{gaunersdorfer1992}.
\end{remark}


\section{Conjectures}
\label{sec:some-open-questions}

We believe that it is possible to use properties of the
measures $\eta_x$  to understand certain dynamical
features of the system involved. 

As a simple example, we observe that $\per(f)=\emptyset$ for
a continuous map $f:\XX\to \XX$ of a compact space, implies
$\eta_x(\XX)<\infty$ for all $x\in\XX$ (e.g., an irrational
circle rotation or torus translation).

Indeed, given $m>1$, every point $y\in\XX$ admits a
neighborhood $U_{y,m}$ such that $U_{y,m}\cap
f^i(U_{y,m})=\emptyset, i=1,\dots,m$ and so
$\eta_x(U_{y,m})<1/m$ for every $x\in\XX$. We obtain an open
cover $\{U_{y,m}\}_{y\in\XX}$ of the compact $\XX$ and so a
finite subcover $U_1,\dots,U_k$ exists. Hence
$\eta_x(\XX)\le\sum_{i=1}^k\eta_x(U_i) \le k/m <\infty$.

Another observation is that if $\eta_x(\XX)<\infty$ for all $x\in\XX$,
then $h_{top}(f)=0$ for a diffeomorphism $f:S\to S$ of a compact
surface or a endomorphism $f:I\to I$ of the circle or interval; and
also for a $C^{1+}$ vector field $G$ on a $3$-manifold.

Indeed, in those cases, 
$h_{top}(f)>0$ ensures the existence of a horseshoe for $f$, or a
suspended horseshoe for the flow $\phi_t$ of $G$ if
$h_{top}(\phi_1)=0$; see \cite{katok80}.  In both cases, these
invariant subsets are conjugated either to a full shift with finitely
many symbols, or to a suspension of such shift; and so if $\eta_x$ is
always a finite measure we contradict
Theorem~\ref{mthm:genericallywild}.

It is then natural to pose the following

\begin{conjecture}\label{conj:htop0}
  Every smooth ($C^{1+}$) diffeomorphism or vector field of a compact
  manifold $\XX$ satisfying $\eta_x(\XX)<\infty$ for all $x\in\XX$ has
  zero topological entropy.
\end{conjecture}



Note that the examples provided by Beguin, Crovisier and Le
Roux in \cite{BegCrovRou07} show that this conjecture is
false if we allow the dynamics to be just a homeomorphism.

It is known in many cases that points with historic behavior form a
geometrically big subset (full Hausdorff dimension) of the dynamics;
see e.g. Barreira-Schmeling \cite{BarSch00}.  Jordan, Naudot and Young
showed in~\cite[Proposition 4.2]{JNY09} that points with ``orbits of
type $B_2$'' carry full topological entropy for the full shift with
finitely many symbols; see Example~\ref{ex:Naudot}. Recently Zhou and
Chen~\cite{ZhCh13a} have show that the set of historic points carries
full topological pressure for systems with non-uniform specification
under certain conditions; this has been generalized by Tian-Varandas
\cite{VarXue}. Here we have shown that wild historic points are
generic in a broad class of examples, so it is natural to pose the
following.

\begin{conjecture}
  \label{conj:fullHD}
  In the class of examples considered in
  Theorem~\ref{mthm:genericallywild}, the set of wild
  historic points has full Hausdorff dimension, full
  topological entropy and full topological pressure.
\end{conjecture}

Since the set $\cH_f\setminus\cW_f$ of historic points which
are not wild is contained in the complement of a generic
subset, that is, $\cH_f\setminus\cW_f$ is meagre, we also
conjecture that this set is small in other ways.

\begin{conjecture}
  \label{conj:notfullHD}
  In the class of examples considered in
  Theorem~\ref{mthm:genericallywild}, \emph{the subset of
    historic points which is not wild is not of full
    Hausdorff dimension, does not carry full topological
    entropy nor full topological pressure}.
\end{conjecture}

Developing our observation about absence of wild historic
behavior after Theorem~\ref{mthm:genericallywild}, we
propose the following.

\begin{conjecture}
  \label{conj:nowild}
  Absence of wild historic points for a smooth enough
  ($C^{1+}$) dynamics of a diffeomorphism or a vector field
  in a compact manifold implies that all invariant
  probability measures are either atomic or have only zero
  Lyapunov exponents.

  An analogous conclusion holds for all smooth enough
  ($C^{1+}$) local diffeomorphisms away from a
  critical/singular set which is sufficiently regular
  (non-flat).
\end{conjecture}

Note that from Peixoto's Theorem
\cite{peixoto62,PM82,gutierrez78,GuBe05} for an open and
dense subset of vector fields in the $C^r$ topology on
compact orientable surfaces, for all $r\ge1$, the limit set
of every orbit is contained in one of finitely many
hyperbolic critical elements (fixed points or periodic
orbits). Hence wild historic points are absent from an open
and dense subset of smooth continuous time dynamics on
surfaces.  Thus, for vector fields
Conjecture~\ref{conj:nowild} makes sense only on manifolds
of dimension $3$ or higher.

We propose the following weakening of Condition (H) from
Theorem~\ref{mthm:openhistoric-heteroclinic}.
\begin{conjecture}
  \label{conj:Hprime}
  For a $C^r$ diffeomorphism of a compact manifold of dimension $2$ or
  higher, $r\ge1$, if there exists a point $x$ and two hyperbolic
  saddle periodic points $P,Q$ and $a,b\ge0, a+b>1$ so that
  \begin{align*}
    \eta_x
    =
    a\cdot\frac1{\pi(P)}\sum_{j=1}^{\pi(P)}\delta_{f^jP}
    +
    b\cdot\frac1{\pi(P)}\sum_{j=1}^{\pi(Q)}\delta_{f^jQ},
  \end{align*}
  where $\pi(P), \pi(Q)$ give the minimal periods of $P, Q$, then $f$
  is accumulated in the $C^r$ topology by diffeomorphisms $g$ so that
  the continuations $P_g, Q_g$ of the periodic points $P,Q$ for the
  diffeomorphism $g$ are homoclinically related.
\end{conjecture}

We note that the modification of Bowen's Example~\ref{ex:Boweneye}
given in \cite{JNY09} with non-hyperbolic saddle points suggest the
following.

\begin{conjecture}\label{conj:Hprimenonhyp}
  The statement of Conjecture~\ref{conj:Hprime} still holds true if we
  remove the hyperbolic assumption on $P,Q$.
\end{conjecture}


\def\cprime{$'$}

\bibliographystyle{abbrv}

\begin{thebibliography}{10}

\bibitem{Al00}
J.~F. Alves.
\newblock {SRB measures for non-hyperbolic systems with multidimensional
  expansion}.
\newblock {\em {Ann. Sci. {\'E}cole Norm. Sup.}}, {33}:{1--32}, {2000}.

\bibitem{ABV00}
J.~F. Alves, C.~Bonatti, and M.~Viana.
\newblock {SRB measures for partially hyperbolic systems whose central
  direction is mostly expanding}.
\newblock {\em {Invent. Math.}}, {140}({2}):{351--398}, {2000}.

\bibitem{AraPac2010}
V.~Ara{\'u}jo and M.~J. Pacifico.
\newblock {\em Three-dimensional flows}, volume~53 of {\em Ergebnisse der
  Mathematik und ihrer Grenzgebiete. 3. Folge. A Series of Modern Surveys in
  Mathematics [Results in Mathematics and Related Areas. 3rd Series. A Series
  of Modern Surveys in Mathematics]}.
\newblock Springer, Heidelberg, 2010.
\newblock With a foreword by Marcelo Viana.

\bibitem{APPV}
V.~Ara{\'u}jo, E.~R. Pujals, M.~J. Pacifico, and M.~Viana.
\newblock Singular-hyperbolic attractors are chaotic.
\newblock {\em Transactions of the A.M.S.}, 361:2431--2485, 2009.

\bibitem{BarrLiVals12}
L.~Barreira, J.~Li, and C.~Valls.
\newblock Full shifts and irregular sets.
\newblock {\em S\~ao Paulo J. Math. Sci.}, 6(2):135--143, 2012.

\bibitem{BarrLiValls14a}
L.~Barreira, J.~Li, and C.~Valls.
\newblock Irregular sets are residual.
\newblock {\em Tohoku Math. J. (2)}, 66(4):471--489, 2014.

\bibitem{BarrLiValls14}
L.~Barreira, J.~Li, and C.~Valls.
\newblock Irregular sets for ratios of {B}irkhoff averages are residual.
\newblock {\em Publ. Mat.}, 58(suppl.):49--62, 2014.

\bibitem{BarrLiValls16}
L.~Barreira, J.~Li, and C.~Valls.
\newblock Irregular sets of two-sided {B}irkhoff averages and hyperbolic sets.
\newblock {\em Ark. Mat.}, 54(1):13--30, 2016.

\bibitem{BarPes2007}
L.~Barreira and Y.~Pesin.
\newblock {\em Nonuniform hyperbolicity}, volume 115 of {\em Encyclopedia of
  Mathematics and its Applications}.
\newblock Cambridge University Press, Cambridge, 2007.
\newblock Dynamics of systems with nonzero Lyapunov exponents.

\bibitem{BarrSau01}
L.~Barreira and B.~Saussol.
\newblock Variational principles and mixed multifractal spectra.
\newblock {\em Trans. Amer. Math. Soc.}, 353(10):3919--3944 (electronic), 2001.

\bibitem{BarSch00}
L.~Barreira and J.~Schmeling.
\newblock Sets of "non-typical" points have full topological entropy and full
  hausdorff dimension.
\newblock {\em Israel Journal of Mathematics}, 116(1):29--70, 2000.

\bibitem{BegCrovRou07}
F.~Beguin, S.~Crovisier, and F.~L. Roux.
\newblock Construction of curious minimal uniquely ergodic homeomorphisms on
  manifolds: the denjoy-rees technique.
\newblock {\em Annales Scientifiques de l'{\'E}cole Normale Sup{\'e}rieure},
  40(2):251 -- 308, 2007.

\bibitem{BL91}
A.~M. Blokh and M.~Y. Lyubich.
\newblock {Measurable dynamics of S-unimodal maps of the interval}.
\newblock {\em {Ann. Sci. {\'E}cole Norm. Sup.}}, {24}:{545--573}, {1991}.

\bibitem{BomfimVarandas15a}
T.~Bomfim and P.~Varandas.
\newblock Multifractal analysis of the irregular set for almost-additive
  sequences via large deviations.
\newblock {\em Nonlinearity}, 28(10):3563, 2015.

\bibitem{BomfimVarandas15}
T.~Bomfim and P.~Varandas.
\newblock Multifractal analysis for weak gibbs measures: from large deviations
  to irregular sets.
\newblock {\em Ergodic Theory and Dynamical Systems}, 37(1):79--102, 2017.

\bibitem{Bo73}
R.~Bowen.
\newblock {Symbolic dynamics for hyperbolic flows}.
\newblock {\em {Amer. J. Math.}}, {95}:{429--460}, {1973}.

\bibitem{Bo75}
R.~Bowen.
\newblock {\em {Equilibrium states and the ergodic theory of Anosov
  diffeomorphisms}}, volume {470} of {\em {Lect. Notes in Math.}}
\newblock {Springer Verlag}, {1975}.

\bibitem{BR75}
R.~Bowen and D.~Ruelle.
\newblock {The ergodic theory of Axiom A flows}.
\newblock {\em {Invent. Math.}}, {29}:{181--202}, {1975}.

\bibitem{DGS76}
M.~Denker, C.~Grillenberger, and K.~Sigmund.
\newblock {\em Ergodic theory on compact spaces}.
\newblock Lecture Notes in Mathematics, Vol. 527. Springer-Verlag, Berlin-New
  York, 1976.

\bibitem{Dowk53}
Y.~Dowker.
\newblock The mean and transitive points of homeomorphisms.
\newblock {\em Annals of Mathematics}, 58(1):123--133, 1953.

\bibitem{gaunersdorfer1992}
A.~Gaunersdorfer.
\newblock {Time averages for heteroclinic attractors}.
\newblock {\em {SIAM J. of Applied Math.}}, {52}:{1476--1489}, {1992}.

\bibitem{GorBoy89}
P.~G{\'o}ra and A.~Boyarsky.
\newblock Absolutely continuous invariant measures for piecewise expanding
  {$C^2$} transformation in {${\bf R}^N$}.
\newblock {\em Israel J. Math.}, 67(3):272--286, 1989.

\bibitem{gutierrez78}
C.~Guti{\'e}rrez.
\newblock {Structural stability for flows on the torus with a cross-cap}.
\newblock {\em {Trans. Amer. Math. Soc.}}, {241}:{311--320}, {1978}.

\bibitem{GuBe05}
C.~Gutierrez and B.~Pires.
\newblock On {P}eixoto's conjecture for flows on non-orientable 2-manifolds.
\newblock {\em Proc. Amer. Math. Soc.}, 133(4):1063--1074, 2005.

\bibitem{Hardy}
G.~H. Hardy.
\newblock Theorems relating to the summability and convergence of slowly
  oscillating series.
\newblock {\em Proceedings of the London Mathematical Society},
  s2-8(1):301--320, 1910.

\bibitem{Hard92}
G.~H. Hardy.
\newblock {\em Divergent series}.
\newblock \'Editions Jacques Gabay, Sceaux, 1992.
\newblock With a preface by J. E. Littlewood and a note by L. S. Bosanquet,
  Reprint of the revised (1963) edition.

\bibitem{HK90}
F.~Hofbauer and G.~Keller.
\newblock {Quadratic maps without asymptotic measure}.
\newblock {\em {Comm. Math. Phys.}}, {127}:{319--337}, {1990}.

\bibitem{JNY09}
T.~Jordan, V.~Naudot, and T.~Young.
\newblock Higher order {B}irkhoff averages.
\newblock {\em Dyn. Syst.}, 24(3):299--313, 2009.

\bibitem{katok80}
A.~Katok.
\newblock {Lyapunov exponents, entropy and periodic orbits for
  diffeomorphisms}.
\newblock {\em {Inst. Hautes {\'E}tudes Sci. Publ. Math.}}, {51}:{137--173},
  {1980}.

\bibitem{KH95}
A.~Katok and B.~Hasselblatt.
\newblock {\em {Introduction to the modern theory of dynamical systems}},
  volume~{54} of {\em {Encyclopeadia Appl. Math.}}
\newblock {Cambridge University Press}, {Cambridge}, {1995}.

\bibitem{KikiLiSo15}
S.~Kiriki, M.-C. Li, and T.~Soma.
\newblock Geometric lorenz flows with historic behavior.
\newblock {\em Discrete and Continuous Dynamical Systems}, 36(12):7021--7028,
  2016.

\bibitem{KrS15}
S.~Kiriki and T.~Soma.
\newblock Takens' last problem and existence of non-trivial wandering domains.
\newblock {\em Advances in Mathematics}, 306:524 -- 588, 2017.

\bibitem{LabRod}
I.~S. Labouriau and A.~A.~P. Rodrigues.
\newblock On {T}akens\cprime last problem: tangencies and time averages near
  heteroclinic networks.
\newblock {\em Nonlinearity}, 30(5):1876, 2017.

\bibitem{LeplYa17}
R.~Leplaideur and D.~Yang.
\newblock {SRB} measure for higher dimensional singular partially hyperbolic
  flows.
\newblock {\em Annales de l'Institut Fourier}, 67(2):2703--2717, 2017.

\bibitem{LiSuTi2013}
C.~Liang, W.~Sun, and X.~Tian.
\newblock Ergodic properties of invariant measures for ${C}^{1+\alpha}$
  non-uniformly hyperbolic systems.
\newblock {\em Ergodic Theory and Dynamical Systems}, 33(2):560--584, 2013.

\bibitem{Man87}
R.~Ma{\~n}{\'e}.
\newblock {\em {Ergodic theory and differentiable dynamics}}.
\newblock {Springer Verlag}, {New York}, {1987}.

\bibitem{MeMor06}
R.~Metzger and C.~Morales.
\newblock Sectional-hyperbolic systems.
\newblock {\em Ergodic Theory and Dynamical System}, 28:1587--1597, 2008.

\bibitem{mtz001}
R.~J. Metzger.
\newblock {Sinai-Ruelle-Bowen measures for contracting Lorenz maps and flows}.
\newblock {\em {Ann. Inst. H. Poincar{\'e} Anal. Non Lin{\'e}aire}},
  {17}({2}):{247--276}, {2000}.

\bibitem{MetzMor15}
R.~J. Metzger and C.~A. Morales.
\newblock Stochastic stability of sectional-anosov flows.
\newblock {\em Preprint arXiv:1505.01761}, {2015}.

\bibitem{niven56}
I.~Niven.
\newblock {\em Irrational numbers}.
\newblock The Carus Mathematical Monographs, No. 11. The Mathematical
  Association of America. Distributed by John Wiley and Sons, Inc., New York,
  N.Y., 1956.

\bibitem{Ols03}
L.~Olsen.
\newblock Extremely non-normal continued fractions.
\newblock {\em Acta Arithmetica}, 108(2):191--202, 2003.

\bibitem{olsen2004}
L.~Olsen.
\newblock Applications of multifractal divergence points to sets of numbers
  defined by their $n$-adic expansion.
\newblock {\em Mathematical Proceedings of the Cambridge Philosophical
  Society}, 136:139--165, 2004.

\bibitem{olsen2004a}
L.~Olsen.
\newblock Applications of multifractal divergence points to some sets of
  d-tuples of numbers defined by their n-adic expansion.
\newblock {\em Bulletin des Sciences Math{\'e}matiques}, 128(4):265 -- 289,
  2004.

\bibitem{Ols04}
L.~Olsen.
\newblock Extremely non-normal numbers.
\newblock {\em Mathematical Proceedings of the Cambridge Philosophical
  Society}, 137(1):43--53, 2004.

\bibitem{olsen-winter07}
L.~Olsen and S.~Winter.
\newblock Multifractal analysis of divergence points of deformed measure
  theoretical birkhoff averages. ii: Non-linearity, divergence points and
  banach space valued spectra.
\newblock {\em Bulletin des Sciences Math{\'e}matiques}, 131(6):518 -- 558,
  2007.

\bibitem{PM82}
J.~Palis and W.~{de Melo}.
\newblock {\em {Geometric Theory of Dynamical Systems}}.
\newblock {Springer Verlag}, {1982}.

\bibitem{peixoto62}
M.~M. Peixoto.
\newblock {Structural stability on two-dimensional manifolds}.
\newblock {\em {Topology}}, {1}:{101--120}, {1962}.

\bibitem{PesPits84}
Y.~B. Pesin and B.~S. Pitskel{\cprime}.
\newblock Topological pressure and the variational principle for noncompact
  sets.
\newblock {\em Funktsional. Anal. i Prilozhen.}, 18(4):50--63, 96, 1984.

\bibitem{Pinho2011}
V.~Pinheiro.
\newblock Expanding measures.
\newblock {\em Annales de l Institut Henri Poincar\'e. Analyse non
  Lin\'eaire,}, 28:899--939, 2011.

\bibitem{PrzyUrb10}
F.~Przytycki and M.~Urba{\'n}ski.
\newblock {\em Conformal fractals: ergodic theory methods}, volume 371 of {\em
  London Mathematical Society Lecture Note Series}.
\newblock Cambridge University Press, Cambridge, 2010.

\bibitem{rogers}
C.~A. Rogers.
\newblock {\em Hausdorff Measures}.
\newblock Cambridge Mathematical Library. Cambridge University Press,
  Cambridge, UK, 2nd edition edition, 1998.

\bibitem{Ro93}
A.~Rovella.
\newblock {The dynamics of perturbations of the contracting Lorenz attractor}.
\newblock {\em {Bull. Braz. Math. Soc.}}, {24}({2}):{233--259}, {1993}.

\bibitem{Ru21}
W.~Rudin.
\newblock {\em {Real and complex analysis}}.
\newblock {McGraw-Hill}, {3} edition, {1987}.

\bibitem{ruelle2001}
D.~Ruelle.
\newblock Historical behaviour in smooth dynamical systems david ruelle.
\newblock In B.~Krauskopf, H.~Broer, and G.~Vegter, editors, {\em Global
  Analysis of Dynamical Systems}, Festschrift dedicated to Floris Takens for
  his 60th birthday, pages 63--65. Taylor \& Franci, 2001.

\bibitem{Sa00}
B.~Saussol.
\newblock {Absolutely continuous invariant measures for multi-dimensional
  expanding maps}.
\newblock {\em {Israel J. Math}}, {116}:{223--248}, {2000}.

\bibitem{shil67a}
L.~P. {Shil'nikov}.
\newblock {The existence of a countable set of periodic motions in the
  neighborhood of a homoclinic curve.}
\newblock {\em {Sov. Math., Dokl.}}, 8:102--106, 1967.

\bibitem{Sm65}
S.~Smale.
\newblock {Diffeomorphisms with many periodic points}.
\newblock In {\em {Differential and Combinatorial Topology (A Symposium in
  Honor of Marston Morse)}}, pages {63--80}. {Princeton Univ. Press},
  {Princeton, N.J.}, {1965}.

\bibitem{Sm67}
S.~Smale.
\newblock {Differentiable dynamical systems}.
\newblock {\em {Bull. Am. Math. Soc.}}, {73}:{747--817}, {1967}.

\bibitem{SuVaYa15}
N.~{Sumi}, P.~{Varandas}, and K.~{Yamamoto}.
\newblock Specification and partial hyperbolicity for flows.
\newblock {\em Dynamical Systems}, 30(4):501--524, 2015.

\bibitem{Ta95}
F.~Takens.
\newblock {Heteroclinic attractors: time averages and moduli of topological
  conjugacy}.
\newblock {\em {Bull. Braz. Math. Soc.}}, {25}:{107--120}, {1995}.

\bibitem{takens08}
F.~Takens.
\newblock Orbits with historic behaviour, or non-existence of averages.
\newblock {\em Nonlinearity}, 21(3):T33, 2008.

\bibitem{thomp10}
D.~Thompson.
\newblock The irregular set for maps with the specification property has full
  topological pressure.
\newblock {\em Dynamical Systems}, 25(1):25--51, 2010.

\bibitem{VarXue}
X.~Tian and P.~Varandas.
\newblock Topological entropy of level sets of empirical measures for
  non-uniformly expanding maps.
\newblock {\em Discrete and Continuous Dynamical Systems}, 37(10):5407--5431,
  2017.

\bibitem{Vi97}
M.~Viana.
\newblock {Multidimensional nonhyperbolic attractors}.
\newblock {\em {Inst. Hautes {\'E}tudes Sci. Publ. Math.}}, {85}:{63--96},
  {1997}.

\bibitem{ZhCh13a}
Z.~Yin, E.~Chen, and X.~Zhou.
\newblock Multifractal analysis of ergodic averages in some non-uniformly
  hyperbolic systems.
\newblock {\em Ergodic Theory and Dynamical Systems}, 36(7):2334--2350, 2016.

\bibitem{ZhCh13}
X.~Zhou and E.~Chen.
\newblock Multifractal analysis for the historic set in topological dynamical
  systems.
\newblock {\em Nonlinearity}, 26(7):1975--1997, 2013.

\end{thebibliography}

\end{document}